\documentclass[11pt,twoside, a4paper, english, reqno]{amsart}
 \usepackage[foot]{amsaddr}
\usepackage[dvips]{epsfig}
\usepackage{amscd}
\usepackage{amssymb}
\usepackage{amsthm}
\usepackage{amsmath}
\usepackage{mathtools}
\usepackage{latexsym}

\usepackage{upref}
\usepackage{bm}
\usepackage{color}
\usepackage{hyperref}
\usepackage{graphicx}
\usepackage{lipsum}
\usepackage{comment}
\usepackage{mathrsfs}

\hypersetup{linkcolor=blue, colorlinks=true
,citecolor = red}

\theoremstyle{plain}
\newtheorem{Th}{Theorem}[section]
\newtheorem{Lem}[Th]{Lemma}
\newtheorem{Cor}[Th]{Corollary}

\theoremstyle{definition}
\newtheorem{Def}[Th]{Definition}

\def\@setemails{%
  \ifnum\theg@author > 1
\mbox{{\itshape E-mail addresses}:\space}{\ttfamily\emails}. \else\mbox{{\itshape E-mail address}:\space}{\ttfamily\emails}. \fi%
}

\newcommand{\be} {\begin{equation}}
\newcommand{\ee} {\end{equation}}

\newcommand{\rt}{\mathbb{R}^{2}}

\newcommand{\inte}{\int_{\rt}}

\newcommand{\F}{\mathcal{F}}

\newcommand{\dw}{\ {\rm d}\omega}

\newcommand{\R} {\mathbb{R}}

\newcommand{\s} {\mathbb{S}}

\newcommand{\authorfootnotes}{\renewcommand\thefootnote{\@fnsymbol\c@footnote}}%

\def\e{{\text{e}}}

\numberwithin{equation}{section} \allowdisplaybreaks

\title[log-HLS and flow under constraint]{A constrained log-HLS inequality and its associated Wasserstein flow}

\author[M. Ghosh]{Monideep Ghosh$^{\dagger}$}

\author[D. Karmakar]{Debabrata Karmakar$^{\ast}$}

\author[R. Sinha]{Reetaj Sinha$^{\ast}$}

\author[K. Verma]{Kartikey Verma$^{\ast}$}

\email{\ttfamily
monideep@isid.ac.in, 
debabrata@tifrbng.res.in, 
reetaj23@tifrbng.res.in}
\email{\ttfamily kartikey24@tifrbng.res.in
}
\address[$\ast$]{Tata Institute of Fundamental Research,
Centre for Applicable Mathematics,
Post Bag No.\ 6503, GKVK Post Office,
Sharada Nagar, Chikkabommasandra,
Bangalore 560065,
Karnataka, India}

\address[$\dagger$]{Theoretical Statistics and Mathematics Unit,
Indian Statistical Institute,
Delhi Centre,
S.J. Sansanwal Marg,
New Delhi, Delhi 110016,
India}



\begin{document}
\begin{abstract}
We study a log-Hardy-Littlewood-Sobolev (log-HLS) inequality under constraints and a naturally induced Wasserstein flow. The constrained log-HLS inequality is dual, in a  variational sense, to the constrained Trudinger-Moser-Onofri-Aubin inequality. The corresponding gradient-flow formulation leads to a parabolic-elliptic Patlak-Keller-Segel (PKS) system whose global dynamics exhibit a dichotomy between global existence and finite-time blow-up depending on the initial mass. Remarkably, the constraint raises the critical mass threshold of the flow by a factor of two compared with the unconstrained setting. This provides a distinct mechanism for suppressing blow-up from those arising in related formulations of the PKS-system, where blow-up is mitigated through additional effects such as a forcing term, coupling with a fluid equation, or the ambient flow generated by harmonic perturbations of the Newtonian potential.\end{abstract}
\maketitle
{\small\tableofcontents}
\section{Introduction}

Let $\s^2$ denote the two-dimensional unit sphere in $\R^3$ equipped with the round metric, and let ${\rm d}\omega$ be the normalized area measure on $\s^2$ such that $\int_{\s^2} {\rm d}\omega = 1$. Let $\nabla_{\s^2}$ and $\Delta_{\s^2}$ denote the Riemannian gradient and the Laplace–Beltrami operator on $\s^2$, respectively. The classical Trudinger–Moser–Onofri \cite{TruY, M-1970, O-1982} inequality asserts that 
\begin{align} \label{TMO}
 J_{\alpha}(u) :=\alpha \int_{\s^2} |\nabla_{\s^2} u|^2 {\rm d} \omega + 2\int_{\s^2} u {\rm d} \omega -\ln{\left(\int_{\s^2}\e^{2u} {\rm d} \omega\right)} \geq 0,
\end{align}
holds for all $u \in H^1(\s^2)$ and all $\alpha \geq 1$. The boundedness from below of $J_{\alpha}$ goes back to Moser \cite{M-1970}, while Onofri \cite{O-1982}, motivated by the study of the determinant of the conformal Laplacian, established that the sharp lower bound is precisely $0$. Other proofs were also given by Hong \cite{Hong-1986} and Osgood-Phillips-Sarnak \cite{OPS-1988}. The value $\alpha = 1$ is critical in the sense that for $\alpha < 1$, the functional $J_\alpha$ is unbounded from below on $H^1(\s^2)$ \cite{M-1970}. For $\alpha > 1$, the minimum value $0$ is achieved only by constant functions. At the critical value $\alpha = 1$, however, there exist non-trivial minimizers, which are completely classified by the conformal transformations of $\s^2$ \cite{O-1982}: $J_1(u) = 0$ if and only if $u = \frac{1}{2}\ln \mathcal{J}_{\psi} + C$ for some conformal map $\psi \colon \s^2 \to \s^2$ and constant $C \in \R$, where $\mathcal{J}_{\psi}$ denotes the Jacobian determinant of $\psi$.

\medskip
In this article we focus on the logarithmic Hardy–Littlewood–Sobolev inequality of Carlen and Loss \cite{CL-1992} and Beckner \cite{B-1993}: Let $G$ be the normalized Green's function of $-\Delta_{\s^2}$ (see subsection \ref{greens}). Then the functional
\begin{align*}
\mathcal{F}_{\alpha}(\rho): = \int_{\s^2}\rho \ln \rho {\dw} - \frac{1}{\alpha}\int_{\s_2}\int_{\s^2} G(\omega,\eta)\left(\rho(\omega)-1\right)\left(\rho(\eta)-1\right)\dw {\rm d} \eta 
\end{align*}
satisfies $\mathcal{F}_{\alpha}(\rho) \geq 0$ for all probability densities $\rho \in \Gamma$, where
\begin{align*}
\Gamma := \left\{\rho \in L^1_+(\s^2) \ | \ \int_{\s^2} \rho \ln \rho \dw < \infty, \int_{\s^2} \rho \dw=1 \right\}
\end{align*}
if and only if $J_{\alpha}(u) \geq 0$ on $H^1(\s^2)$. In particular, $\inf_{\rho \in \Gamma} \mathcal{F}_{\alpha}(\rho) \geq 0$ if and only if $\alpha \geq 1$, while $\inf_{\rho \in \Gamma} \mathcal{F}_{\alpha}(\rho) = -\infty$ whenever  $\alpha < 1$. 
Carlen, Carrillo and Loss \cite{CCLHLS} gave a short proof of the sharp log-HLS inequality in $\R^2$ via the fast-diffusion flow. Recently Dolbeault and Li \cite{DolIMRN} introduced a generalized version with an external logarithmic potential, covering both the classical and a reverse regime based on an extension of the method proposed in \cite{CCLHLS}.

\medskip

The functional $J_\alpha$ may be regarded as the primal formulation on $H^1(\mathbb{S}^2)$, and $\mathcal{F}_\alpha$ as its dual formulation on $\Gamma$. Formally, the passage from the primal variable $u$ to the dual variable $\rho$ is realised by the Gibbs principle 
 $\rho = \e^{2u} / \int_{\s^2} \e^{2u} \,{\rm d}\omega$. 
 Conversely, given a density $\rho\in\Gamma$, the corresponding primal variable $u$ is recovered as the Newtonian potential
  $u(\cdot) = \int_{\s^2} G(\cdot, \omega)\big(\rho(\omega) - 1\big) \,{\rm d}\omega$ or equivalently by solving $-\Delta_{\s^2} u = \rho - 1$ (see \cite[Theorem 5]{CL-1992}, \cite[Section 3]{SW}, and Section \ref{Gibbs section} below).
  
  \medskip
  
 Another interesting duality principle was established by Agueh, Boroushaki, and Ghoussoub \cite[Theorem 3.1 and Corollary 3.5]{ABG2017} using optimal mass transport. Building on the energy-entropy duality developed by Cordero, Nazaret, and Villani \cite{ENV2004} and generalized by Agueh, Ghoussoub, and Kang \cite{GKGAFA2004}, they established in \cite{ABG2017} a duality formula on bounded balls in $\mathbb{R}^2$, relating an equivalent Onofri functional on $\mathbb{R}^2$ to a renormalized free-energy functional with reference density $\mu_2(x)=[\pi(1+|x|^2)^2]^{-1}$. An earlier attempt by Dolbeault, Esteban, and Jankowiak \cite{DEJ2015} to prove the Onofri inequality via mass transport was restricted to radial functions. As noted in \cite{ABG2017}, this duality reflects a common free-energy formulation underlying both the two-dimensional Moser-Onofri inequality and the critical Sobolev inequalities in higher dimensions. Consequently, the prescribed Gaussian curvature equation and the Yamabe equation are both linked to stationary solutions of the rescaled fast-diffusion equation, with the two-dimensional case formulated on bounded domains subject to Neumann boundary conditions.
 
\medskip

After Moser’s work, Aubin \cite{A-1979} observed that the critical threshold for $\alpha$ can be lowered under a center-of-mass (first-moment) constraint: the functional $J_\alpha$ remains bounded from below on the subspace of $H^1(\mathbb{S}^2)$ consisting of functions satisfying $\int_{\mathbb{S}^2}\omega\,e^{2u}\,d\omega=0$
if and only if $\alpha\ge\frac{1}{2}$.
Chang and Yang \cite{CY-1987} subsequently showed, in their influential work on the prescribed Gaussian curvature problem, that the same center-of-mass constraint also preserves the sharpness of the lower bound near the critical threshold: for all $\alpha$ sufficiently close to $1$, one has $\inf J_\alpha=0$ on the constrained subspace of $H^1(\mathbb{S}^2)$.

\medskip

The sharp form of this result was established nearly three decades later by Gui and Moradifam \cite{GM-2018}. Their work completed a long line of partial results and conjectures originating in Chang and Yang’s article \cite[Prop.~B]{CY-1987}, which were motivated by the search for minimizing conditions on the curvature function in the Gaussian curvature prescription problem on $\mathbb{S}^2$. Significant intermediate progress was achieved by Feldman et al.~ \cite{FFGG-1998}, Gui–Wei \cite{GW-2000}, C.~S.~Lin \cite{Lin1-2000}, and Ghoussoub–Lin \cite{GL-2010}. A comprehensive account of these developments can be found in the monograph of Ghoussoub and Moradifam \cite{GM-2013}.

\medskip

The higher-dimensional analogue of the Onofri inequality was established by Beckner \cite{B-1993}. In the spirit of Aubin’s observation \cite{A-1979}, Beckner's inequality likewise admits an improvement under a first-moment (center-of-mass) constraint \cite{BCYCMP92,CY1995AnnMath}. Determining the sharp lower bound of the associated family of functionals under this constraint has become known as the generalized Chang–Yang conjecture, motivated by the two-dimensional results of Chang and Yang \cite{CY-1987}. The conjecture has attracted considerable attention in recent years; see, for instance, \cite{WeiXu98, ImBec1, ImBec2, ImBec3, ImBec4}. In a very recent work, Gui-Li-Wei-Ye \cite{ImBec5} resolved for axially symmetric functions for every dimension bigger or equal to $ 3$.

\medskip

In this article, one of our major goals is to establish a dual formulation of the Moser-Aubin inequality. Namely, we ask whether the functional $\mathcal{F}_{\alpha}$ is bounded from below over
\begin{align*}
\Gamma_0 := \left\{\rho \in \Gamma \ | \ \int_{\s^2} \omega\rho \ {\rm d}\omega = 0\right
\}
\end{align*}

for $\alpha \geq \frac{1}{2}.$ As we will see, the answer is positive. We prove the following main theorem:

\begin{Th}\label{main1}
For every $\alpha \geq \frac{1}{2},$ $\inf_{ \rho \in \Gamma_0} \mathcal{F}_{\alpha}(\rho) = 0.$ 
Moreover,  the infimum is attained only by the constant function $\rho \equiv 1.$
\end{Th}

We recently learned that Andreasson, Berman, and Svensson independently established this result in the second version of \cite{logHLS}, having originally posed it as a conjecture in the first. While we gratefully acknowledge their precedence, our work offers a distinctly different methodology and context. Specifically, whereas their proof relies on Legendre-Fenchel duality, ours employs a concentration–compactness approach in the spirit of Chang-Hang \cite{CH-2022}. A notable advantage of our method is its potential applicability to problems with higher-moment constraints, as discussed below in the present context, and already demonstrated by Chang and Hang for Moser's inequality. For the classification of optimizers, both works follow the classical framework, utilizing the Kazdan–Warner condition \cite{KW-1971} and the uniqueness theorem of Gui-Moradifam \cite{GM-2018}.

\medskip

In a recent article, Chang and Gui \cite{CG-2023} consider the following constraint-free functional

\begin{align}\label{Chang-GuiIneq}
I_{\alpha}(u) = \alpha\int_{\s^2}|\nabla_{\s^2}u|^2 {\rm d}\omega+2 \int_{\s^2} u {\rm d}\omega -\frac{1}{2}\ln\left[\left(\int_{\s^2}e^{2u}{\rm d}\omega\right)^2-\sum_{i=1}^3\left(\int_{\s^2}\omega_i e^{2u}{\rm d} \omega\right)^2\right],
\end{align}
on $H^1(\mathbb{S}^2)$, and ask for the optimal $\alpha$ such that $I_{\alpha}(u) \geq 0$ for all $u \in H^1(\mathbb{S}^2)$. In this case, the optimal $\alpha$ is actually $\frac{2}{3}$, as opposed to the  anticipated value $\frac{1}{2}.$ 
Moreover, the functional is not bounded from below if $\alpha < \frac{2}{3}.$ In fact, this inequality was our starting point for improving the log-HLS inequality under moment constraints. Our second main result is a dual Chang-Gui inequality, which translates to the log-HLS inequality under moment constraints. 

\medskip

We obtain a Gibbs principle involving $\ln\left[\left(\int_{\s^2}e^{2u}{\rm d}\omega\right)^2-\sum_{i=1}^3\left(\int_{\s^2}\omega_i e^{2u}{\rm d} \omega\right)^2\right]$  instead of $\ln \int_{\s^2} e^{2u}{\rm d} \omega$ (as in the classical Gibbs principle), which helps us prove the following general log-HLS inequality incorporating the first moment:

\begin{Th}[general log-HLS under constraint] \label{main2}
For every $\alpha\geq \frac{2}{3}, \bm{a} \in B(0,1)$ and $\rho\in \Gamma$ satisfying $\int_{\s^2}\omega \rho(\omega) \ {\rm d}\omega = \bm{a}$, we have,
\begin{align*}&\int_{\s^2}\frac{1-\bm{a}\cdot\omega}{1-|\bm{a}|^2}\rho \ln \rho {\dw}-\frac{1}{2}\ln(1-|\bm{a}|^2) \\ \geq & \, \frac{1}{\alpha}\int_{\s_2}\int_{\s^2} G(\omega,\eta)\left(\frac{1-\bm{a}\cdot\omega}{1-|\bm{a}|^2}\rho(\omega)-1\right)\left(\frac{1-\bm{a}\cdot\eta}{1-|\bm{a}|^2}\rho(\eta)-1\right)\dw {\rm d} \eta.
\end{align*}
\end{Th}

In Section \ref{equivalence section}, we will prove that the Chang-Gui inequality \eqref{Chang-GuiIneq} and the above general log-HLS inequality are equivalent. We remark that, as a corollary of Theorem \ref{main2}, the log-HLS functional $\mathcal{F}_{\alpha}$ is bounded from below by $0$ for all $\rho \in \Gamma$ satisfying $\int_{\s^2}\omega \rho(\omega) \ {\rm d}\omega = 0$, provided $\alpha \geq \frac{2}{3}.$ Therefore, Theorem \ref{main1} is stronger than this particular case of Theorem \ref{main2}. Here, the key point is that, when the center of mass is allowed to vary in $B(0,1)$, the optimal value of $\alpha$ is $\frac{2}{3}.$
The proof of Theorem \ref{main1} relies on the concentration-compactness principle in the spirit of Chang and Hang \cite{CH-2022}, which also provides the inspiration for seeking further improvements under higher-moment constraints.

\medskip

Chang and Hang \cite{CH-2022} discovered a family of improved Trudinger-Moser inequalities under higher moment constraints. To state their result, let us recall the following notations from their article. For any non-negative integer $k$, we denote
\begin{align*}
    \mathscr{P}_k &= \{ \text{all polynomials on } \mathbb{R}^3 \text{ with degree at most } k \}; 
    \mathring{\mathscr{P}}_k = \left\{ p \in \mathscr{P}_k : \int_{\mathbb{S}^2} p \, \dw = 0 \right\};  \\
    \mathscr{H}_k &= \{ \text{all degree } k \text{ homogeneous polynomials on } \mathbb{R}^3 \}; \ 
    \mathring{\mathscr{H}}_k = \{ h \in \mathscr{H}_k : \Delta_{\mathbb{R}^3} h = 0 \}. 
\end{align*}
Then 
\begin{equation*}
   \mathscr{P}_k|_{\mathbb{S}^2} = \bigoplus_{i=0}^k \mathring{\mathscr{H}}_i|_{\mathbb{S}^2} \ \ \ \ \mbox{where} \ \ \ \ \mathring{\mathscr{H}}_k|_{\mathbb{S}^2} = \left\{ h|_{\mathbb{S}^2} : h \in \mathring{\mathscr{H}}_k \right\}. 
\end{equation*}
For $m \in \mathbb{N}$, denote
\begin{align}
    \mathcal{N}_m &= \left\{ N \in \mathbb{N} : \ \mbox{there exists} \ x_i \in \mathbb{S}^2 \text{ and } \lambda_i \in (0, 1), \ 1 \leq i \leq N \ \text{s.t. } \ \sum_{i=1}^N \lambda_i = 1 \right. \notag \\     
    &\quad \left. \text{ and for any } p \in \mathring{\mathscr{P}}_m, \sum_{i=1}^N\lambda_i p(x_i) = 0 \right\} 
\end{align}
and define, $N_m = \min \mathcal{N}_m$.
The quantity $N_m$ plays the role of the optimal constant in the Trudinger-Moser inequality of Chang and Hang under constraints: for any $\epsilon > 0$, there exists a constant $c_{\epsilon}$ such that

\begin{equation}
    \log \int_{\mathbb{S}^2} e^{2u} \, {\rm d}\omega \leq \left( \frac{1}{N_m} + \epsilon \right) \int_{\s^2} |\nabla_{\s^2} u|^2 \, {\rm d}\omega+ c_\epsilon,
\end{equation}
holds for all $u \in H^1(\mathbb{S}^2)$ satisfying $\int_{\mathbb{S}^2} u \,  {\rm d}\omega = 0$ and $\int_{\mathbb{S}^2} p e^{2u} \,  {\rm d}\omega = 0$ for every $p \in \mathring{\mathscr{P}}_m.$ In addition, the value $\frac{1}{N_m}$ is optimal in the sense that no such finite $c_{\epsilon}$ exists if one replaces the coefficient of the gradient term by $\frac{1}{N_m} - \epsilon.$ To the best of our knowledge, the exact value of $N_m$ is known only for small values of $m$. In general, however, the lower bound $N_m \geq \left(\left [ \frac{m}{2}\right]+1\right)^2$ is known to hold \cite{Hesse2015}, where $[\kappa]$ denotes the largest integer less than or equal to $\kappa$.

\medskip

Adopting the concentration-compactness principle of Chang and Hang \cite{CH-2022}, we also studied the log-HLS inequality under higher-order constraints. Although we were not able to prove the expected optimal result, we establish the following in this regard. 
\begin{Th} \label{main3}
  Let $\alpha > \frac{1}{\left(\left [ \frac{m}{2}\right]+1\right)^2}$, then  there exists a constant $C(\alpha)> -\infty$ such that $\mathcal{F}_{\alpha}(\rho) > -C(\alpha)$ for every $\rho \in \Gamma$ satisfying $\int_{\s^2}p(\omega) \rho(\omega) \dw = 0$ for every $p \in \mathring{\mathscr{P}}_m$.
\end{Th}

We remark that the dimension of $\oplus_{k=0}^m\mathring{\mathscr{H}}_k := (m+1)^2,$ and the quantity appeared in Theorem \ref{main3} is precisely one over the appropriate dimension. We also remark that for $m = 2,$ the result is indeed almost optimal as in this case $N_2 = 4.$  

\medskip

Next, we study an application of the above improved log-HLS inequality by considering the constrained Wasserstein flow of the associated functional.

\subsection{A constrained Wasserstein flow: Patlak-Keller-Segel system under constraint}

Lifting the functional $\mathcal{F}_{\alpha}$ corresponding to $\alpha = 1$ to $\R^2$ via stereographic projection yields the classical log-HLS functional
\begin{align*}
\F_{\chi}(\rho) = \inte \rho(x) \ln \rho(x) \ {\rm d}x + \frac{\chi}{4\pi} \inte \inte \rho(x) \ln |x-y| \rho(y) \ {\rm d}x{\rm d}y,
\end{align*}
with $\chi = 8\pi.$ We keep the same notation $\rho$ for the densities, which we hope will not cause any confusion to the reader. Since we treat $\s^2$ and $\R^2$ in separate sections, there will be no confusion regarding their domains of definition.

\medskip

It follows from the independent works of Carlen-Loss \cite{CL-1992} and Beckner \cite{B-1993} that $\F_{\chi}$ is bounded from below on the set of all $\rho$ with $\int_{\R^2} \rho = 1$ if and only if $\chi = 8\pi.$ In fact, both articles established their inequality as a limiting case of the HLS inequality in all dimensions. As a consequence of Theorem \ref{main1}, we will deduce in Section \ref{loghls section}, Theorem \ref{loghlsr2} that $\mathcal{F}_{\chi}$ is bounded from below over 
\begin{align*}
\overset{\circ}{\Gamma}&:=\left \{\rho\in L^1_{+}(\R^2):\int_{\R^2}\rho\ln \rho \ {\rm d} x <\infty,\int_{\R^2}|x|^2\rho  \ {\rm d} x  <\infty, \int_{\R^2} \mathcal{S}(x) \rho(x) \ {\rm d} x = 0 \right \}
\end{align*}
 provided $\chi \leq 16\pi$, where $\mathcal{S}(x) = \frac{1}{1 + |x|^2}(2x, |x|^2 - 1)$  is the stereographic projection. 

\medskip

Let $\mathcal{P}_2(\R^2)$ be the set of all Borel probability measures with finite second moment $\int_{\R^2} \vert{}x\vert{}^2 d\mu(x) < \infty,$ and let $\mathcal{P}_{\mbox{\tiny{ac,2}}}(\R^2)$ be the subset of these measures that are absolutely continuous with respect to the Lebesgue measure. Let $W_2(\mu, \nu)$ be the Wasserstein distance between $\mu$ and $\nu \in \mathcal{P}_2(\R^2),$ obtained from the optimal transportation problem with the quadratic cost $\vert{}x - y\vert{}^2$ (see Section \ref{Wflow section}). The formal $W_2$-gradient flow, as discovered by Otto \cite{Otto} (see also \cite{JKO}), is then given by
\begin{align} \label{formal gf}
\partial_t \rho = \nabla \cdot \left(\rho \nabla \mathcal{F}_{\chi}^{\prime}(\rho)\right)
\end{align}
where $ \mathcal{F}_{\chi}^{\prime}(\rho)$ is the first variation of $\mathcal{F}_{\chi}$ with respect to $\rho.$ Equation \eqref{formal gf}
is the the classical Patlak-Keller-Segel system for chemotaxis: 
\begin{equation}\label{KS}
\begin{cases}
\partial_t \rho(x,t) = \Delta \rho - \chi \nabla \cdot \left[\rho(x,t) \nabla u(x,t) \right], \ \ \mbox{in} \ \R^2 \times (0,T) \\ 
u(x,t) := - \frac{1}{2\pi}\int_{\mathbb{R}^2}\ln |x - y| \rho(y,t) {\rm d}y, \ \ \mbox{in} \ \R^2 \times (0,T).
\end{cases}
\end{equation}
We refer to \cite{Pa, KS, Horsurvey2, W1, Horsurvey1} for its origin, biological relevance, and historical motivations of such a systems. One of the interesting facts about \eqref{KS} is that it admits an optimal critical parameter $\chi_{\star} = 8\pi.$ If $\chi \leq 8 \pi,$ there exists a global weak solution to \eqref{KS}, while if $\chi > 8\pi,$ all solutions blow up in finite time \cite{BD, BCM}. This system has been extensively studied, with particular emphasis on phenomena such as the global existence, long-time asymptotic, finite time blow-up profiles,  and the suppression of blow-up through the introduction of additional terms \cite{Childress, NSenba, SeSu, SSbook, Subook, BT, BKLN, BD, BCM, BDP, BCCjfa, CD, FM, GMas, Weietal}. In this article, we focus on this latter topic: the suppression of blow-up.

\medskip
 
 In \cite{HeTadmorKS} (see also \cite{HeTCAMS}), He and Tadmor considered a more realistic scenario which takes into account the flow of the ambient environment induced by a harmonic potential, and thus retaining the elliptic structure as in the original PKS system.
The harmonic perturbation considered is of the form $A (-x_1^2 + x_2^2).$ 
They demonstrate that for symmetric initial configurations where the mass is concentrated sufficiently far from the axis, a suitably large amplitude $A$ ensures global existence, provided $\chi < 16\pi$. The strong ambient field acts as a repelling force that drives the concentrated mass away from the axis, counteracting chemotactic attraction and preventing finite-time blow-up.

 This interesting result motivated us to consider the constrained Wasserstein flow of $\mathcal{F}_{\chi}$ for $\chi \leq 16\pi,$ which preserves the constraint $\int_{\R^2} \mathcal{S}(x) \rho \ {\rm d}x = 0$ along the flow. In Section \ref{Wflow section}, we formally derive such a flow for a general functional. Focusing on the log-HLS functional, we obtain the following PKS-system under constraints:

 \begin{equation} \label{KSc}
\begin{cases}
\partial_t \rho(x,t) = \Delta \rho - \chi \nabla \cdot \left[\rho(x,t) \nabla\left(u(x,t) + \frac{1}{\chi}\bm{\lambda}(t) \cdot \mathcal{S}(x)\right) \right] \\ 
u(x,t) := - \frac{1}{2\pi}\int_{\mathbb{R}^2}\ln |x - y| \rho(y,t) {\rm d}y\\ 
\mathcal{A}[\rho(t)](\bm{\lambda}(t)) = - \bm{b}[\rho(t)], 
\end{cases}
\end{equation}
with initial data $\rho(x,0) = \rho_0(x)$. Here $\mathcal{A}[\rho] \in \mathbb{R}^{3\times 3}$ and $\bm{b}[\rho] \in \R^3$ whose components are given by 
\begin{align*}
\left(\mathcal{A}[\rho]\right)_{ij} = \int_{\mathbb{R}^2}\rho \nabla \mathcal{S}_i \cdot \nabla \mathcal{S}_j \ {\rm d}x,
\end{align*}
\begin{align*} 
 \bm{b}[\rho]_i  
 = -\int_{\R^2}   \rho\Delta \mathcal{S}_i   + \frac{\chi}{4\pi} \int_{\R^2} \int_{\R^2} (\nabla \mathcal{S}_i(x) - \nabla \mathcal{S}_i(y)) \cdot \frac{x - y}{\vert{}x - y\vert{}^2} \rho(x) \rho(y) \,{\rm d}x \,{\rm d}y,
\end{align*}
and $\bm{\lambda}(t)\in \R^3$ is a time dependent Lagrange multiplier. While the system appears complicated, we will see in subsection \ref{pks subsection} that these expressions arise quite naturally in constrained Wasserstein flow. Note that since $\nabla \mathcal{S}_i(x)$ and $\Delta \mathcal{S}_i(x)$ are bounded, $\mathcal{A}[\rho]$ and $\bm{b}[\rho]$ are well defined for integrable densities. In fact, we will show that the matrix $\mathcal{A}[\rho]$ is positive definite for every $\rho$ having finite entropy, ensuring that $\bm{\lambda}(t)$ is uniquely determined by $\rho$.

\medskip

Along the flow \eqref{KSc}, the functional $\F_{\chi}$ is decreasing. Thus, it serves as a Lyapunov functional for the system; moreover, the moment constraint $\int_{\R^2} \mathcal{S}(x) \rho(x,t) \ {\rm d}x = 0$ is preserved.

\medskip

We humbly acknowledge that we are unaware of any natural biological process that is naturally governed by the present constrained system \eqref{KSc}. In particular, the Lagrange multiplier enforcing the constraint depends globally on the density distribution, which does not appear to have an obvious counterpart in standard biological mechanisms. We therefore regard the constraint as a deliberately imposed mathematical device, conceivable perhaps in a controlled or highly engineered setting, and study the resulting system primarily for its intrinsic mathematical interest.

\medskip

In Section \ref{jko section}, we employ the JKO-scheme \cite{JKO} and show that it converges to a weak solution (see Definition \ref{weak formulation}) of the system. We briefly state our result here and refer the reader to Sections \ref{Wflow section} and \ref{jko section} for a detailed discussion.

\begin{Th} \label{main4}
Let $\chi < 16\pi$ and $\rho_0 \in \overset{\circ}{\Gamma}$. Then the PKS-system \eqref{KSc} admits a global weak solution.
\end{Th}

\medskip

The article is organized as follows. In Section \ref{Gibbs section}, we recall the Gibbs variational principle for entropy and deduce the Gibbs principle incorporating first-moment constraints. In Section \ref{equivalence section}, using this principle, we prove the general log-HLS inequality (Theorem \ref{main2}) and demonstrate its equivalence to the Chang-Gui inequality \cite{CG-2023}, thereby establishing its sharpness. In Section \ref{loghls section}, we study the optimal log-HLS inequality under a first-moment constraint and classify its extremals (Theorem \ref{main1}). Later in the same section, we prove the improved log-HLS inequality under higher-moment constraints, establishing Theorem \ref{main3}. In Section \ref{Wflow section}, we introduce a constrained Wasserstein flow and deduce the constrained PKS system described above. Finally, in Section \ref{jko section}, we implement the minimizing movement scheme (or JKO scheme) to show that the constrained PKS system admits a global weak solution for $\chi < 16\pi$, proving Theorem \ref{main4}.


\section{Gibbs variational principle for entropy} \label{Gibbs section}
The classical Gibbs variational principal for entropy assess that 
\begin{align*}
\int_{\s^2} \rho \ln \rho \dw = \sup_{v \in C^{\infty}(\s^2), \bar v = 0}\left[ \int_{\s^2} \rho v \dw - \ln \left(\int_{\s^2} e^v \dw\right)\right],
\end{align*}
where recall $\dw$ is the normalized measure on $\s^2$. Equivalently, for every $u \in H^1(\s^2),$ the following formula holds
\begin{align*}
\ln \left(\int_{\s^2} e^{u} \ {\rm d}\omega \right) = \sup_{\rho \in L^1_+(\s^2), \ \int_{\s^2} \rho \dw = 1} \left\{ \int_{\s^2} u \rho \dw - \int_{\s^2} \rho\ln \rho \dw\right\}.
\end{align*}

The variational representations above offer a beautiful dual characterization. While the expressions for entropy and free energy were well-known in statistical mechanics from the time of Gibbs and Boltzmann, the identity long remained a piece of analytical folklore within the information theory community. The classical reference is the seminal work of Donsker and Varadhan \cite{DV1975}, which, to the best of our knowledge, provided the first rigorous proof of this identity over general spaces. Without delving into further details, for this particular case, it is not difficult to establish these inequalities purely from a variational point of view. The common strategy is to show that an extremum exists, write down the corresponding Euler-Lagrange equation, and evaluate the functional at the critical point.

\medskip

In this section, we propose a variation of the above Gibbs principle. This variational principle takes into account the centre of mass of the density. Before we proceed let us do some elementary algebra. We write
\begin{align}\label{log rewrite}
&\left[\left(\int_{\s^2} e^{u} \ {\rm d}\omega \right)^2 - \left|\left(\int_{\s^2} \omega e^{u} \ {\rm d}\omega \right) \right|^2\right] \notag\\
=& \int_{\s^2}\int_{\s^2} e^{u(\omega) + u(\eta)}  \ {\rm d}\omega {\rm d}\eta - \int_{\s^2} \int_{\s^2}\omega \cdot \eta e^{u(\omega) + u(\eta)}  \ {\rm d}\omega {\rm d}\eta \notag\\
=&\int_{\s^2}\int_{\s^2} (1 - \omega \cdot \eta)e^{u(\omega) + u(\eta)}  \ {\rm d}\omega {\rm d}\eta.
\end{align}
Let $\mathcal{P}(\s^2\times\s^2)$ be set of all probability measures on the product space $\s^2 \times \s^2$ and define the measure 
\begin{align*}
d\pi_{\star}(\omega, \eta) = (1 - \omega \cdot \eta)\ {\rm d}\omega {\rm d}\eta.
\end{align*}

By Gibbs variational principle on the product spaces \cite[Lemma 2.1]{DV1975} we obtain from \eqref{log rewrite}

\begin{align} \label{gibbs product}
&\ln \left[\left(\int_{\s^2} e^{u} \ {\rm d}\omega \right)^2 - \left|\left(\int_{\s^2} \omega e^{u} \ {\rm d}\omega \right) \right|^2\right] \notag\\
=& \sup_{\pi \in \mathcal{P}(\s^2\times\s^2)} \left[\int_{\s^2}\int_{\s^2} (u(\omega) + u(\eta))\ {\rm d}\pi(\omega, \eta) - \int_{\s^2}\int_{\s^2} \ln \left(\frac{{\rm d}\pi}{{\rm d}\pi_{\star}}\right)\ {\rm d}\pi(\omega, \eta)\right].
\end{align}
Here, the last term is the relative entropy with respect to the measure $\pi_{\star}$  defined by 
\begin{align*}
\int_{\s^2}\int_{\s^2} \ln \left(\frac{{\rm d}\pi}{{\rm d}\pi_{\star}}\right)\ {\rm d}\pi = 
\begin{cases}
\int_{\s^2}\int_{\s^2} f\ln f \ {\rm d}\pi_{\star}, \ \ \mbox{if} \ \pi \ll \pi_{\star}, {\rm d}\pi = f {\rm d}\pi_{\star},\\
+\infty, \ \ \ \hspace{2 cm}  \mbox{otherwise}.
\end{cases}
\end{align*}
We recall $\Gamma := \left\{\rho \in L^1_+(\s^2) \ | \ \int_{\s^2} \rho \ln \rho \dw < \infty, \int_{\s^2} \rho \dw=1 \right\}$ and define the subset with fixed center of mass $\bm{a} \in B(0,1)$ in $\R^3:$
\begin{align*}
\Gamma_{\bm{a}} := \left\{\rho \in \Gamma \ | \  \int_{\s^2} \omega\rho \dw=\bm{a} \right\}.
\end{align*}

For every $\bm{a} \in B(0,1), u \in H^1(\s^2)$ and $\rho \in \Gamma_{\bm{a}}$ we also  define the function with domain $H^1(\s^2)\times \Gamma_{\bm{a}}$ by
\begin{align*}
\Phi_{\bm{a}}(u,\rho) = \frac{1}{1 - |\bm{a}|^2} \left[2\int_{\s^2} (1 - \bm{a} \cdot \omega) u \rho \dw - 2 \int_{\s^2}(1 - \bm{a} \cdot \omega) \rho \ln \rho \dw\right] + \ln (1 - |\bm{a}|^2).
\end{align*}
We also note that for every $\rho \in \Gamma_{\bm{a}},\Phi_{\bm{a}}$ satisfies the following  homogeneity
\begin{align*}
\Phi_{\bm{a}}(u + c,\rho) = \Phi_{\bm{a}}(u,\rho) + 2c, \ \ \mbox{for every} \ c \in \mathbb{R}.
\end{align*}
 
First we prove the following  Gibbs principle.
\begin{Th}[Gibbs variational principle for entropy]\label{GVP}
For every $u \in H^1(\s^2)$ we have
\begin{align*}
\ln \left[\left(\int_{\s^2} e^{u} \ {\rm d}\omega \right)^2 - \left|\left(\int_{\s^2} \omega e^{u} \ {\rm d}\omega \right) \right|^2\right]
= \sup_{\bm{a} \in B(0,1)} \sup_{\rho \in \Gamma_{\bm{a}}} \Phi_{\bm{a}}(u,\rho).
\end{align*}
\end{Th}

\begin{proof}
Fix $u \in H^1(\s^2)$. Let $\bm{a} \in B(0,1),$ and $\rho \in \Gamma_{\bm{a}},$ and define  
\begin{align*}
{\rm d}\pi (\omega,\eta) = \frac{1 - \omega \cdot \eta}{1 - |\bm{a}|^2} \rho(\omega)\rho(\eta) \dw {\rm d}\eta.
\end{align*}
We claim that the quantity with in the supremum in \eqref{gibbs product} is precisely the function $\Phi_{\bm{a}}(u,\rho).$ For that goal we simplify the expressions involved:
 \begin{align*}
 \int_{\s^2}\int_{\s^2} (u(\omega) + u(\eta))\ {\rm d}\pi(\omega, \eta) 
 =& \frac{1}{1- |\bm{a}|^2} \int_{\s^2}\int_{\s^2} (u(\omega) + u(\eta))(1 - \omega \cdot \eta)\rho(\omega)\rho(\eta) \dw {\rm d}\eta\\
 =& \frac{2}{1 - |\bm{a}|^2} \int_{\s^2} u(\omega)\rho(\omega) \int_{\s^2} (1 - \omega \cdot \eta) \rho(\eta) \dw {\rm d}\eta \\
 =& \frac{2}{1 - |\bm{a}|^2} \int_{\s^2} (1 - \bm{a} \cdot \omega)u \rho \dw.
 \end{align*}
 On the other hand, the relative entropy can be simplified as follows:
 \begin{align*}
 &\int_{\s^2}\int_{\s^2} \ln \left(\frac{{\rm d}\pi}{{\rm d}\pi_{\star}}\right)\ {\rm d}\pi(\omega, \eta) \\
 =&\frac{1}{1 - |\bm{a}|^2}\int_{\s^2}\int_{\s^2} \ln \left(\frac{\rho(\omega)\rho(\eta)}{1 - |\bm{a}|^2}\right)  (1 - \omega \cdot \eta) \rho(\omega)\rho(\eta) \dw {\rm d}\eta \\
 =&\frac{1}{1 - |\bm{a}|^2}\left[\int_{\s^2}\int_{\s^2}(1 - \omega \cdot \eta) \rho(\omega) \rho(\eta) \Big(\ln \rho(\omega) + \ln \rho(\eta)- \ln (1 - |\bm{a}|^2)\Big)\right]\\
 =& \frac{1}{1 - |\bm{a}|^2}\left[2\int_{\s^2}\int_{\s^2}(1 - \omega \cdot \eta) \rho(\omega)\rho(\eta)\ln \rho(\omega) \dw {\rm d}\eta
 - (1 - |\bm{a}|^2)\ln (1 - |\bm{a}|^2)\right]\\
 =&  \frac{2}{1 - |\bm{a}|^2}\int_{\s^2} \rho(\omega)\ln \rho(\omega) \dw\int_{\s^2}(1 - \omega \cdot \eta)\rho(\eta) {\rm d}\eta - \ln(1 - |\bm{a}|^2)\\
 =& \frac{2}{1 - |\bm{a}|^2}\int_{\s^2} (1 - \bm{a} \cdot \omega )\rho\ln \rho \dw - \ln (1 - |\bm{a}|^2).
 \end{align*}
As a result we deduce 
\begin{align*}
\left[\int_{\s^2}\int_{\s^2} (u(\omega) + u(\eta))\ {\rm d}\pi(\omega, \eta) - \int_{\s^2}\int_{\s^2} \ln \left(\frac{{\rm d}\pi}{{\rm d}\pi_{\star}}\right)\ {\rm d}\pi(\omega, \eta)\right] = \Phi_{\bm{a}}(u, \rho).
\end{align*}
Hence by \eqref{gibbs product} we deduce
\begin{align*}
\ln \left[\left(\int_{\s^2} e^{u} \ {\rm d}\omega \right)^2 - \left|\left(\int_{\s^2} \omega e^{u} \ {\rm d}\omega \right) \right|^2\right]
\geq \sup_{\bm{a} \in B(0,1)} \sup_{\rho \in \Gamma_{\bm{a}}} \Phi_{\bm{a}}(u,\rho).
\end{align*}
holds. To prove the opposite inequality, using the homogeneity of both terms terms and adding a constant, we can assume $\int_{\s^2} e^u \dw = 1.$ Let $\bm{a} = \int_{\s^2}\omega e^u \dw \in B(0,1)$ be the center of mass of $u.$ We set $\rho = e^u.$ Then it is easy to verify that for this choice of $\rho$,
$\Phi_{\bm{a}}(u, \rho) = \ln (1 - |\bm{a}|^2),$ precisely what needed to be proved.
\end{proof}

As an immediate corollary, applying it to the zero center of mass, we derive
\begin{Cor}
For every $u \in H^1(\s^2)$ we have
\begin{align*}
\ln \left[\left(\int_{\s^2} e^{2u} \ {\rm d}\omega \right)^2 - \left|\left(\int_{\s^2} \omega e^{2u} \ {\rm d}\omega \right) \right|^2\right]
\geq \sup_{\rho \in \Gamma_0} \left[ 4 \int_{\s^2} u \rho \dw - 2 \int_{\s^2}\rho \ln \rho \dw\right].
\end{align*}
\end{Cor}

\section{Equivalence of Chang-Gui and the General log-HLS inequality: Theorem \ref{main2} } \label{equivalence section}

In this section, we prove a general log-HLS inequality incorporating the center of mass of the density. Naturally, because of this constraint, one would expect an improved log-HLS inequality with a lower value of $\alpha < 1.$ Here, we show that if the center of mass is allowed to take any value in $B(0,1),$ then the best lower bound for $\alpha$ is $\frac{2}{3}.$ In the next section, we further improve this to the optimal value $\frac{1}{2}$ by requiring the center of mass to be at the origin.

\subsection{Green's Function of $-\Delta_{\s^2}$} \label{greens}

We denote the Green's function with respect to the normalized measure as $G(\omega,\eta)$, which satisfies
\begin{align*}
-\Delta_{\s^2}G(\cdot,\eta)=\delta_{\eta}-1
\end{align*}
in the sense of distributions. Thus, for any $\phi\in C^{\infty}(\s^2)$,
\begin{align*}
-\int_{\s^2} G(\omega,\eta)\Delta_{\s^2} \phi(\omega){\rm d}\omega=\phi(\eta)-1.
\end{align*}
Explicitly, $G(\omega,\eta)$ is given by
\begin{align*}
G(\omega,\eta)=-2\ln|\omega-\eta|+ \mbox{const}.
\end{align*}
where the additive constant is fixed by $\int_{\s^2}G(\omega,\eta) \dw=0$.

\subsection{General and Improved Logarithmic Hardy-Littlewood-Sobolev Inequality}
\begin{Th}[General Logarithmic Hardy-Littlewood-Sobolev Inequality]\label{GLHLS}
For every $\alpha\geq \frac{2}{3}$ and $\rho\in \Gamma_{\bm{a}}$, we have,
\begin{align} \label{gloghls}
&\int_{\s^2}\frac{1-\bm{a}\cdot\omega}{1-|\bm{a}|^2}\rho \ln \rho {\dw}-\frac{1}{2}\ln(1-|\bm{a}|^2) \notag \\ 
\geq &\frac{1}{\alpha}\int_{\s_2}\int_{\s^2} G(\omega,\eta)\left(\frac{1-\bm{a}\cdot\omega}{1-|\bm{a}|^2}\rho(\omega)-1\right)\left(\frac{1-\bm{a}\cdot\eta}{1-|\bm{a}|^2}\rho(\eta)-1\right)\dw {\rm d} \eta.
\end{align}  
\end{Th}
\begin{proof}
By Chang-Gui inequality \cite[Theorem 2.1]{CG-2023} and Gibbs variational principle for entropy Theorem \ref{GVP}, for $\alpha\geq \frac{2}{3}$, 
\begin{align*}
&\alpha\int_{\s^2}|\nabla_{\s^2} u|^2{\rm d} \omega+2\int_{\s^2}u{\rm d}\omega\\
\geq &\frac{1}{2}\ln\left[\left(\int_{\s^2} \e^{2u}{\rm d} \omega\right)^2-\left|\int_{\s^2}\omega \e^{2u}{\rm d}\omega\right|^2\right]\\
\geq&
\frac{1}{1-|\bm{a}|^2}\left[2\int_{\s^2}u\rho(1-\bm{a}\cdot \omega){\rm d} \omega-\int_{\s^2}(1-\bm{a}\cdot\omega )\rho\ln \rho {\rm d}\omega\right]+\frac{1}{2}\ln(1-|\bm{a}|^2).
\end{align*}
for every $\rho\in \Gamma_{\bm{a}}$.
Hence,
\begin{align*}
\int_{\s^2} \left[\frac{1-\bm{a}\cdot\omega}{1-|\bm{a}|^2}\rho\ln\rho\dw\right]-\frac{1}{2}\ln(1-|\bm{a}|^2)\geq 2\int_{\s^2}u\left[\rho\frac{1-\bm{a}\cdot \omega}{1-|\bm{a}|^2}-1\right]{\dw}-\alpha\int_{\s^2}|\nabla u|^2\dw. 
\end{align*}

Since L.H.S. is independent of $u$, the R.H.S. is uniformly bounded above for every $u\in H^1(\s^2)$.  Moreover, when $\|u\|_{H^1(\s^2)}\to \infty$, the R.H.S. tends to $ -\infty$. Therefore by direct method in calculus of variation we conclude that R.H.S. admits a maximizer in $H^1(\s^2)$, that we denote by $u$ itself. 

 \medskip
 
We set $\hat{\rho}(\omega)=\frac{1-\bm{a}\cdot\omega}{1-|\bm{a}|^2}\rho$. Then the Euler-Lagrange equation for $u$ satisfies
\begin{align*}
-\alpha\Delta_{\s^2} u= \hat{\rho}-1.
\end{align*}
Hence the maximum value of the R.H.S. is  
\begin{align*}
\frac{1}{\alpha}\int_{\s^2}\int_{\s^2} G(\omega,\eta)(\hat{\rho}(\omega)-1)(\hat{\rho}(\eta)-1){\rm d}\omega {\rm d}\eta,
\end{align*}
completing the proof. 
\end{proof}
As a direct corollary we have the following log-HLS inequality under constraints:
\begin{Cor}[Improved Logarithmic Hardy-Littlewood-Sobolev Inequality]\label{ILHLS}
For every $\alpha\geq \frac{2}{3}$ and $\rho\in \Gamma_0$, we have,
\begin{align*}
\int_{\s^2}\rho \ln \rho {\dw}
\geq \frac{1}{\alpha}\int_{\s_2}\int_{\s^2} G(\omega,\eta)\left(\rho(\omega)-1\right)\left(\rho(\eta)-1\right)\dw {\rm d} \eta.
\end{align*}  

\end{Cor}

\subsection{An Equivalence}

Next we show that the Chang-Gui inequality \cite[Theorem 2.1]{CG-2023} and the log-HLS inequality given in Theorem \ref{GLHLS} are equivalent. In particular, it will prove that  $\alpha = \frac{2}{3}$ is optimal. 

\begin{Th}\label{equivalent}
The following are equivalent:
\begin{itemize}
\item[(a)] Chang-Gui Inequality:
For $\alpha\geq \frac{2}{3}$ and $u \in H^1(\s^2)$, 
\begin{align*}
\alpha\int_{\s^2}|\nabla_{\s^2} u|^2\dw + 2\int_{\s^2}u \dw-\frac{1}{2}\ln\left[\left(\int_{\s^2}\e^{2u}\dw\right)^2-\left|\int_{\s^2}\omega \e^{2u}\dw\right|^2\right]\geq 0.
\end{align*}
\item[(b)] General log-HLS inequality \eqref{gloghls} holds, for every $\bm{a} \in B(0,1).$
\end{itemize}
\end{Th}
\begin{proof}
The implication $(a)\implies (b)$ is proved in Theorem \ref{GLHLS}. We prove the opposite  implication.

\medskip

Given $u\in H^1(\s^2)$, we define $\rho=\frac{\e^{2u}}{\int_{\s^2}\e^{2u}\dw}$ and  $\bm{a}=\int_{\s^2}\omega \rho\dw.$ For the simplification of the proof we assume without loss of generality, $\int_{\s^2}\e^{2u}\dw=1$.
We define $\hat{\rho}=\frac{1-\bm{a}\cdot\omega}{1-|\bm{a}|^2}\rho$ and $v(\omega)=\frac{1}{\alpha}\int_{\s^2}G(\omega,\eta)(\hat{\rho}(\eta)-1) {\rm d}\eta$. Then 
\begin{align} \label{EEq1}
-\alpha\Delta_{\s^2}v=(\hat{\rho}-1).
\end{align}
Multiplying \eqref{EEq1} by $v$ and integrating by parts, we get
\begin{align*}
\alpha \int_{\s^2}|\nabla_{\s^2} v|^2\dw =\int_{\s^2} (\hat{\rho}-1)v\dw.
\end{align*}
Therefore, the General Logarithmic Hardy-Littlewood-Sobolev inequality gives,
\begin{align}\label{EEq2}
2\int_{\s^2} u \hat{\rho}\dw- \alpha\int_{\s^2}|\nabla_{\s^2} v|^2\dw\geq \frac{1}{2}\ln(1-|\bm{a}|^2).
\end{align} 
We will maximize the L.H.S..
For this we again integrate by part the equation \eqref{EEq1} but with the test function $u$ to obtain
\begin{align*}
\mbox{L.H.S.}&=2\int_{\s^2}u \hat{\rho}\dw-\alpha\int_{\s^2} |\nabla_{\s^2} v|^2\dw\\
&=2\int_{\s^2}u(\hat{\rho}-1)-\alpha\int_{\s^2}|\nabla_{\s^2} v|^2\dw +2\int_{\s^2}u\dw\\
&=2\alpha\int_{\s^2}\langle\nabla_{\s^2} u, \nabla_{\s^2} v\rangle\dw-\alpha \int_{\s^2}|\nabla_{\s^2}v|^2\dw +2\int_{\s^2}u \dw\\
&=\alpha\int_{\s^2} |\nabla_{\s^2} u|^2\dw-\alpha \int_{\s^2}| \nabla_{\s^2} (u-v)|^2\dw+2 \int_{\s^2} u\dw\\
&\leq \alpha\int_{\s^2}|\nabla_{\s^2} u|^2\dw +2\int_{\s^2}u \dw. 
\end{align*}
Combining the above inequality with the inequality \eqref{EEq2}, and recalling that  $\int_{\s^2}\e^{2u}\dw=1$, we obtain the desired implication. 
\end{proof} 

Theorem \ref{GLHLS} and Theorem \ref{equivalent} proves Theorem \ref{main2} and it's sharpness.

\section{Optimal Log-HLS inequality with center of mass constraint} \label{loghls section}

 In this section, we prove Theorem \ref{main1} that is for $\alpha\geq \frac{1}{2}$ the constrained infimum to be finite with best lower bound to be zero.
 For the convenience of the reader, we recall that the log-HLS functional is defined by
\begin{align*}
\mathcal{F}_{\alpha}(\rho)=\int_{\s^2}\rho\ln \rho\ \dw -\frac{1}{\alpha}\int_{\s^2}\int_{\s^2} G(\omega,\eta)(\rho(\omega)-1)(\rho(\eta)-1)\ \rm d\omega\rm d\eta,
\end{align*}
over the admissible set of center of mass at $0$
\begin{align*}
\Gamma_0 := \left\{\rho \in L^1_+(\s^2) \ | \ \int_{\s^2} \rho \ln \rho \dw < \infty, \int_{\s^2} \rho \dw=1,  \int_{\s^2} \omega\rho \dw=0 \right\}.
\end{align*}

We define 
\begin{align*}
u(\omega)=\int_{\s^2} G(\omega,\eta)(\rho(\eta)-1){\rm d} \eta =: G\star(\rho-1)(\omega).
\end{align*} 
 Then, $\bar{u} : = \int_{\s^2} w \dw=0$ and $u$ solves the PDE
 $$-\Delta_{\s^2}u=\rho-1.$$
We denote the entropy and the interaction energy as follows:
\begin{align*}
\mathcal{H}(\rho) :=\int_{\s^2}\rho\ln \rho\  {\rm d}\omega, \ \ \mathcal{E}(\rho):=\int_{\s^2}\int_{\s^2} G(\omega,\eta)(\rho(\omega)-1)(\rho(\eta)-1)\dw {\rm d}\eta=\int_{\s^2}|\nabla_{\s^2} u|^2\ {\rm d}\omega.
\end{align*}

We will achieve the optimal $\alpha$ as follows: first we prove that for $\alpha > \frac{1}{2},$ $\mathcal{F}_{\alpha}$ is bounded from below on $\Gamma_0.$ In the next step we show that for every $\alpha > \frac{1}{2}$ the infimum is attained and hence solves an appropriate Euler-Lagrange equation. Because of the constraints, there will be Lagrange multipliers. In the next step we show that the Lagrange multiplies corresponding to the moment constraint vanishes, and hence we can invoke the result of Gui-Moradifam \cite{GM-2018} to conclude that $\rho \equiv 1$ is the only minimizers and hence $\inf_{\Gamma_0} \mathcal{F}_{\alpha} = 0$ for all $\alpha > \frac{1}{2}.$ Therefore, the same inequality persists to the critical threshold $\alpha = \frac{1}{2}.$ 
\medskip

\begin{Th} \label{proof alpha half}
For every $\alpha > \frac{1}{2},$ the infimum $\inf_{ \rho \in \Gamma_0} \mathcal{F}_{\alpha}(\rho) > -\infty.$
\end{Th}

\begin{proof}
{\bf Step 1}: Suppose there exists a sequence $\{\rho_{n}\}$ such that $\int_{\s^2}\rho_n\dw=1, \int_{\s^2}\omega\rho_n\dw = 0$
and 
$\mathcal{F}_{\alpha}(\rho_n)\to -\infty$. We set $u_n=G\star(\rho_n-1)$ then and $\bar{u}_n=0$ and 
\begin{align*}
\mathcal{F}_{\alpha}(\rho_n)=\int_{\s^2} \rho_n\ln \rho_n\dw -\frac{1}{\alpha}\int_{\s^2}\rho_n u_n
=\mathcal{H}(\rho_n)-\frac{1}{\alpha}\mathcal{E}(\rho_n).
\end{align*}

\medskip

  We claim that $\mathcal{H}(\rho_n) \to \infty $, and $\mathcal{E}(\rho _n)\to \infty$.

\medskip

If $\mathcal{E}(\rho_n)\leq C$ then $\mathcal{F}_{\alpha}(\rho_n)\to -\infty$ would give $\mathcal{H}(\rho_n)\to -\infty$. This is impossible as $\mathcal{H}(\rho_n)\geq 0$. Similarly, if $\mathcal{H}(\rho_n)\leq C$, by log-HLS inequality \cite{CL-1992},  $\mathcal{E}(\rho_n)\leq \mathcal{H}(\rho_n)\leq C.$

\medskip

 We set  $m_n=\sqrt{\mathcal{E}(\rho_n)}
$ and $w_n=\frac{u_n}{m_n}$.
 Then, 
 $$\|\nabla w_n\|_2=1\ \ \ ,\ \ \ \ \bar{w}_n=0$$ 
 
 and
\begin{align*}
\int_{\s^2}w_n \rho_n\dw =\frac{1}{m_n}\int_{\s^2} u_n\rho_n \dw 
=\frac{1}{m_n}\mathcal{E}(\rho_n)= m_n.
\end{align*}
 The second equality follows from $\int_{\s^2}u_n \dw=0.$

 \medskip

 {\bf Step 2:} By Gibbs variational formula we have,
 \begin{align*}
 \ln \int_{\s^2}\e^{2m_n w_n}\dw&\geq 2 m_n\int_{\s^2} w_n \rho_n-\int_{\s^2} \rho_n\ln \rho_n\\
 &=2 m_n^2- \mathcal{H}(\rho_n).
 \end{align*}
 
\medskip
 
Since $\mathcal{F}_{\alpha}(\rho_n)\to -\infty$, we have for large $n$, $\mathcal{H}(\rho_n)\leq \frac{1}{\alpha}\mathcal{E}(\rho_n)$. Therefore,
\begin{align} \label{CH log bound}
\ln\int_{\s^2}\e^{2m_nw_n}\geq \left(2-\frac{1}{\alpha}\right)m_n^2
\end{align}
with $m_n\to \infty$, $\|\nabla w_n\|=1$. Applying the concentration compactness lemma of Chang-Hang \cite[Proposition 2.4]{CH-2022}, we conclude 
\begin{align*}
w_n\rightharpoonup w \ \ \mbox{in}\ \ \ H^1(\s^2)\ \ \ ,\ \ \  |\nabla w_n|^2\ \dw \overset{\star}{\rightharpoonup} |\nabla w|^2\ \dw+\sigma, \ \mbox{in measure}
\end{align*}
and
\begin{align} \label{a weak conv}
\frac{\e^{2m_n w_n}}{\int_{\s^2} \e^{2m_nw_n}\dw}\overset{\star}{\rightharpoonup} \nu\ \ \  \mbox{in measure}, \ \nu=\sum_{i=1}^{N}\nu_i \delta_{x_i},\ \ \  \nu_i\geq 0, \ \ \ \sum_{i=1}^N\nu_i=1,
\end{align}
where the blow-up set is characterized by  $\left\{ x\in \s^2: \sigma(\{x\})\geq (2-\frac{1}{\alpha})\right\}=\{x_1,\dots x_N\}.$

\medskip

{\bf Step 3:} Let $\rho_n \dw \overset{\star}{\rightharpoonup} \mu$ in measure, then it satisfies $\int_{\s^2} \omega{\rm d}\mu = 0.$ We claim that $\mu = \frac{1}{2}(\delta_{\omega_0} + \delta_{-\omega_0})$ for some $\omega_0 \in \s^2.$

To prove it, let $\omega_0 \in \ \mbox{supp}(\mu)$ be such that $\mu(\{\omega_0\}) < \frac{1}{2}.$ Choose $r_0$ small enough such that $\mu(B_{r_0}(\omega_0)) < \frac{1}{2}.$  

For $\omega \in \s^2$ and $r >0,$ let $m(\rho_n, B_r(\omega)) := \int_{B_r(\omega)} \rho_n$ denote the mass of $\rho_n$ in the ball $B_r(\omega).$ Shrinking $r_0$ if necessary, we can assume that there exists a $\delta>0$ such that for large enough $n,$ 
\begin{align} \label{a mass bound}
\delta \leq m(\rho_n, B_{r_0}(\omega)) < \frac{1}{2} - \delta, \ \ \ \mbox{for all} \ \ \ \omega \in B_{\frac{r_0}{2}}(\omega_0).
\end{align}
 Now for $\omega \in B_{\frac{r_0}{2}}(\omega_0),$ we estimate $u_n.$ Recall that the Green's function satisfies the bound: 
 $G(\omega, \eta) \leq -2 \ln d(\omega, \eta) + C.$ Using this bound, we estimate
 \begin{align*}
 u_n(\omega) &= \int_{\s^2} G(\omega, \eta) \rho_n(\eta) {\rm d}\eta \\
 & = \int_{B_{r_0}(\omega)} G(\omega, \eta) \rho_n(\eta) {\rm d}\eta + \int_{\s^2 \backslash B_{r_0}(\omega)} G(\omega, \eta) \rho_n(\eta) {\rm d}\eta \\
& \leq - 2 \int_{B_{r_0}(\omega)} \ln d(\omega, \eta) \rho_n(\eta) {\rm d\eta} + C \\
&= - 2 m(\rho_n, B_{r_0}(\omega)) \int_{B_{r_0}(\omega)} \ln d(\omega, \eta) \frac{\rho_n(\eta)}{m(\rho_n, B_{r_0}(\omega))} {\rm d\eta} + C.
 \end{align*}
Applying Jensen's inequality, we get
\begin{align*}
e^{2 u_n(\omega)} \leq C \int_{B_{r_0}(\omega)}  d(\omega, \eta)^{- 4 m(\rho_n, B_{r_0}(\omega))} \frac{\rho_n(\eta)}{m(\rho_n, B_{r_0}(\omega))} {\rm d\eta}, \ \ \mbox{for all} \ \omega \in B_{\frac{r_0}{2}}(\omega_0).
\end{align*}
Integrating, and using \eqref{a mass bound} we deduce, $\int_{B_{\frac{r_0}{2}}(\omega_0)} e^{2 u_n(\omega)} \dw \leq C.$ 
By \eqref{CH log bound} and \eqref{a weak conv}, we conclude that $\omega_0 \not \in \ \mbox{supp}(\nu).$ Hence, 
$\mbox{supp}(\nu) \subset \{\omega \ : \ \mu(\{\omega\}) \geq \frac{1}{2}\}.$ This concludes simultaneously that $\{\omega \ : \ \mu(\{\omega\}) \geq \frac{1}{2}\} \neq \emptyset$ and that because of the center of mass constraint, $\mu$ has to be the sum of Dirac masses as claimed.

\medskip

{\bf Step 4:} From Step 3, it is classical to show that $\mathcal{F}_{\alpha}(\rho_n)$ remains bounded from below. For reader's convenience, we outline the proof. 

Recall that $G(\omega,\eta) = -2\log d(\omega,\eta) + c$, where $d$ denotes the geodesic distance on $\s^2$ and $c$ is a fixed constant. 

We decompose the sphere into two disjoint hemispheres, $\s^2=U_+\cup U_-,$
$U_+\cap U_-=\emptyset,$
separated by the equator. We choose the decomposition so that the concentration point $\omega_0$ lies in $U_+$ (at the north pole), and consequently $-\omega_0\in U_-$ (at the south pole).

We decompose $\rho_n$ in to two parts: $\rho_n = \rho_{1,n} + \rho_{2,n} := \rho_n \chi_{U_+} + \rho_n \chi_{U_-}$, then
\begin{align*}
\rho_{1,n} \rightharpoonup \frac{1}{2} \delta_{\omega_0}, \qquad \rho_{2,n} \rightharpoonup  \frac{1}{2} \delta_{-\omega_0}, 
\end{align*}
weakly as measures on $\s^2$. Fix $\varepsilon > 0$.
We decompose the interaction energy in to three parts: $\mathcal{E}(\rho_n) =  \mathcal{E}(\rho_{1,n}) + \mathcal{E}(\rho_{2,n}) + \mathcal{E}_{\mbox{\tiny{cross}}}(\rho_n).$

Then by classical log-HLS inequality
\begin{align*}
\frac{1}{\alpha}\left(\mathcal{E}(\rho_{1,n}) + \mathcal{E}(\rho_{2,n}) \right)
&\leq \frac{ \int_{\s^2} \rho_{1,n} \dw}{\alpha}\mathcal{H}(\rho_{1,n}) + \frac{ \int_{\s^2} \rho_{2,n} \dw}{\alpha}\mathcal{H}(\rho_{2,n}) + C \\
& \leq \frac{\max\{ \int_{\s^2} \rho_{1,n} \dw, \int_{\s^2} \rho_{2,n} \dw \}}{\alpha}\mathcal{H}(\rho_{n}) + C \\
& \leq (1 - \varepsilon)\mathcal{H}(\rho_{n}) + C,
\end{align*}
since $\alpha > \frac{1}{2},$ and where $C$ is independent of $n.$ Therefore it is enough to estimate the cross term as below:
\begin{align*}
\mathcal{E}_{\mbox{\tiny{cross}}}(\rho_n) := 2\int_{\s^2} \int_{\s^2}G(\omega,\eta)\, \rho_{1,n}(\omega)\, \rho_{2,n}(\eta) \, {\rm d}\omega\, {\rm d}\eta \leq \varepsilon \mathcal{H}(\rho_{n}) + C.
\end{align*}

Since $G$ is smooth outside the diagonal, it is enough to prove estimate of the form:
\begin{align} \label{cross term}
\int\int_{\{d(\omega, \eta) < \delta\}}G(\omega,\eta)\, \rho_{1,n}(\omega)\, \rho_{2,n}(\eta) \, {\rm d}\omega\, {\rm d}\eta \leq \varepsilon \mathcal{H}(\rho_{n}) + C.
\end{align}
for some $\delta >0.$ Fix $\delta$ small such that
\begin{align*}
\{d(\omega,\eta) < \delta\} \;=\; \Bigl(\{d(\omega,\eta)<\delta\} \cap \{\eta \notin B_{\delta}(- \omega_0)\}\Bigr) \;\cup\; \Bigl(\{d(\omega,\eta)<\delta\} \cap \{\omega \notin B_{\delta}(\omega_0)\}\Bigr).
\end{align*}

Fix $\beta \in (1, 2).$ Using $G(\omega,\eta) \leq -2\log d(\omega,\eta) + C,$
and Young's inequality $ab \le b\log b + e^a,$
with
\begin{align*}
a = \beta \log \frac{1}{d(\omega,\eta)}, \qquad b = \beta^{-1} \rho_{1,n}(\omega),
\end{align*}
we obtain, on the set $\{d(\omega,\eta)<\delta\} \cap \{\eta \notin B_{\delta}(- \omega_0)\}$,
\begin{align*}
\rho_{1,n}(\omega) \log \frac{1}{d(\omega,\eta)} \le \beta^{-1} \rho_{1,n}(\omega) \log \rho_{1,n}(\omega) + C + d(\omega,\eta)^{-\beta}.
\end{align*}

Since, $\int_{\{d(\omega,\eta)<\delta\}} d(\omega,\eta)^{-\beta} \, {\rm d}\omega \le C_{\beta} < \infty,$ 
integrating against $\rho_{2,n}(\eta)\, {\rm d}\eta$ over $\{\eta \notin B_\varepsilon(-\omega_0)\}$ gives
\begin{align*}
&\int\int_{\{d(\omega,\eta)<\delta\} \cap \{\eta \notin B_{\delta}(- \omega_0)\}} G(\omega,\eta)\, \rho_{1,n}(\omega)\, \rho_{2,n}(\eta) \, {\rm d}\omega\, {\rm d}\eta \\
&\leq \frac{2}{\beta} \left( \int_{\s^2 \setminus B_{\delta}(-\omega_0)} \rho_{2,n}\, {\rm d}\eta \right) \int_{\s^2} \rho_{1,n} \log \rho_{1,n} \, {\rm d}\omega +  C \leq \varepsilon \mathcal{H}(\rho_{1,n}) + C.
\end{align*}
Here the last inequality follows from the fact that $\rho_{2,n}$ has small mass in $\s^2 \setminus B_{\delta}(-\omega_0).$
The estimate over the other region is similar. This proves \eqref{cross term} and hence $\mathcal{F}_{\alpha}(\rho_n)$ remains uniformly bounded below. 
\end{proof}

It is now classical to show that $\inf_{ \rho \in \Gamma_0} \mathcal{F}_{\alpha}(\rho) > -\infty$ is attained for every $\alpha > \frac{1}{2}.$ We fix a $\alpha > \frac{1}{2},$ and denote the corresponding optimizer by $\rho$ (the dependence on $\alpha$ has been ignored). Then $\rho$ satisfies the Euler-Lagrange equation 
\begin{equation} \label{rho EL}
\log\rho(\omega)-\frac{2}{\alpha}\int_{\s^2}G(\omega,\eta)\bigl(\rho(\eta)-1\bigr)\,{\rm d}\eta = a+b\cdot\omega.
\end{equation}
where $a\in\mathbb{R}$, $b\in\mathbb{R}^3$ are Lagrange multipliers. 

\medskip 

We define $u(\omega):=\frac{1}{\alpha} \int_{\s^2}G(\omega,\eta)\bigl(\rho(\eta)-1\bigr)\,{\rm d}\eta$ (slightly different than the earlier notation where $u$ was defined with out the $\alpha$). Then $u$ solves 
 $-\Delta u=\frac1\alpha(\rho-1)$, i.e.
\begin{equation} \label{the EL log}
-\alpha \Delta u + 1=\rho =   \frac{e^{2u+b\cdot\omega}}{ \int_{\s^2} e^{2u+b\cdot\eta} {\rm d}\eta}.
\end{equation}
Integrating \eqref{the EL log} against ${\rm d}\omega$ and $\omega {\rm d}\omega$ and using $\rho \in \Gamma_0$, $-\Delta_{\s^2} \omega = 2 \omega$, we conclude that 
\begin{align} \label{a u property}
\int_{\s^2} u \ {\rm d}\omega = \int_{\s^2}  \omega u \ {\rm d}\omega = 0.
\end{align}
By classical regularity theory, we obtain that $u$ and $\rho$ are smooth.
Next we show that all the Lagrange multiplier vanishes. 

\begin{Lem} \label{multiplier vanishes}
Let $\rho \in \Gamma_0, u \in H^1(\s^2)$ satisfies the EL equation \eqref{the EL log}, then $b = 0.$
\end{Lem}

\begin{proof}
For $c \in \mathbb{R}^3,$ we define $F(\omega) = c \cdot \omega.$ Then we recall the following result from Kazdan and Warner \cite[equation (8.3)]{KW-1971}
\begin{equation} \label{KW condition}
\int_{\s^2}\Delta_{\s^2} u (\nabla_{\s^2} F \cdot \nabla_{\s^2} u) \,{\rm d}\omega=0. 
\end{equation}

Using $\nabla_{\s^2}(b\cdot\omega)=b-(b\cdot\omega)\omega$, and taking derivative in \eqref{rho EL}, we get
\begin{align*}
\nabla_{\s^2}\rho=\rho\Bigl(2\nabla_{\s^2} u+b-(b\cdot\omega)\omega\Bigr).
\end{align*}

Let $\omega = (\omega_1, \omega_2, \omega_3),$ then taking inner product with $\nabla_{\s^2} \omega_i,$ we get
\begin{equation}
\nabla_{\s^2} \rho \cdot \nabla_{\s^2} \omega_i=\rho\Bigl(2 \nabla_{\s^2} u \cdot \nabla_{\s^2} \omega_i+b\cdot\nabla_{\s^2}\omega_i-(b\cdot\omega)\,\omega\cdot\nabla_{\s^2}\omega_i\Bigr).
\end{equation}
Rearranging the terms, we get 
\begin{equation} \label{an identity}
\rho \nabla_{\s^2} u \cdot \nabla_{\s^2} \omega_i=\tfrac12 \nabla_{\s^2} \rho \cdot \nabla_{\s^2} \omega_i -\tfrac12\rho\,b\cdot\nabla_{\s^2}\omega_i+\tfrac12\rho(b\cdot\omega)\,\omega\cdot\nabla_{\s^2}\omega_i. 
\end{equation}
 Since $\Delta_{\s^2}\omega_i =-2\omega_i$ and $\int\omega_i\rho\,{\rm d}\omega=0$, integration by parts gives 
\begin{align*}
\int_{\s^2}  \nabla_{\s^2} \rho \cdot \nabla_{\s^2} \omega_i \ {\rm d}\omega=2\int_{\s^2}\omega_i\rho\,{\rm d}\omega=0.
\end{align*}
On the other hand, using the identities $\nabla_{\s^2}\omega_i=e_i-\omega_i\omega$ equivalently $b\cdot\nabla_{\s^2}\omega_i=b_i-(b\cdot\omega)\omega_i,$ and $\omega\cdot\nabla_{\s^2}\omega_i=0,$
and integrating \eqref{an identity} we conclude
\begin{equation}\label{last id 1}
\int_{\s^2} \rho \nabla_{\s^2} u \cdot \nabla_{\s^2} \omega_i \,{\rm d}\omega=-\tfrac12\Bigl[b_i-\int\rho(b\cdot\omega)\omega_i\,{\rm d}\omega\Bigr]. 
\end{equation}
On the other hand, from \eqref{the EL log}, $\rho=1 - \alpha \Delta_{\s^2} u$, and so
\begin{align} \label{last id 2}
\int_{\s^2} \rho \nabla_{\s^2} u \cdot \nabla_{\s^2} \omega_i \,{\rm d}\omega = \int_{\s^2}  \nabla_{\s^2} u \cdot \nabla_{\s^2} \omega_i \,{\rm d}\omega - \alpha\int_{\s^2} \Delta u (\nabla_{\s^2} u \cdot \nabla_{\s^2} \omega_i )\,{\rm d}\omega  
= 0,
\end{align}
where the first term vanishes by \eqref{a u property} and the second term vanishes by the Kazdan Warner condition \eqref{KW condition}. Hence \eqref{last id 1} and \eqref{last id 2} gives
\begin{align*}
b_i=\int_{\s^2}\rho\,(b\cdot\omega)\,\omega_i\,{\rm d}\omega, \qquad i=1,2,3,
\end{align*}
i.e.\ in vector form,
\begin{equation} \label{the matrix}
b= \left(\int_{\s^2}\rho\,(\omega\otimes\omega)\,{\rm d}\omega \right) b.
\end{equation}
It is easy to verify that the above matrix is positive definite.
Since the trace of the matrix $\left(\int_{\s^2}\rho\,(\omega\otimes\omega)\,{\rm d}\omega \right)$ is $\int_{\s^2} \rho \ {\rm d}\omega = 1,$ if $b \neq 0,$ \eqref{the matrix} shows that $b$ is an eigenvector corresponding to the eigenvalue $1$, and hence the other two eigen values much be zero. This contradiction proves $b = 0.$
\end{proof}

As a consequence of the above results, we can now state and prove the main theorem. 

\begin{Th}
For every $1 > \alpha \geq \frac{1}{2},$ $\inf_{ \rho \in \Gamma_0} \mathcal{F}_{\alpha}(\rho) = 0.$ Moreover, the infimum is attained only by the constant function $\rho \equiv 1.$ 
\end{Th}

\begin{proof}
Let $\alpha > \frac{1}{2}.$ Then we know that the infimum is attained by some $\rho \in \Gamma_0.$ As before, we set $u(\omega):=\frac{1}{\alpha} \int_{\s^2}G(\omega,\eta)\bigl(\rho(\eta)-1\bigr)\,{\rm d}\eta$, then by Lemma \ref{multiplier vanishes}, $u$ solves the PDE
\begin{equation} \label{optimal sol eqn}
-\alpha \Delta u + 1=\rho =   \frac{e^{2u}}{ \int_{\s^2} e^{2u} {\rm d}\eta}.
\end{equation}
By classification result of Gui and Moradifam \cite[Theorem 4.3]{GM-2018}, the only solution to \eqref{optimal sol eqn} is $u \equiv 0$ and hence $\rho \equiv 1.$ This proves $\mathcal{F}_ {\alpha} (\tilde \rho) \geq 0$ for every  $\tilde \rho \in \Gamma_0$ and $\alpha > \frac{1}{2},$ and hence the same inequality persists for $\alpha = \frac{1}{2}.$ The classification result of Gui and Moradifam holds up to $\alpha = \frac{1}{2}$ and hence the same argument gives that if a minimizer exists then it must be identically equal to $1$. 
\end{proof}

\noindent
{\bf Optimality:} Let $\alpha < 1/2$. Given $N \in \mathbb{N}$. Choose a small cap $C \subset \mathbb{S}^2$ centered at the North Pole and a density $\rho_1^*$ of mass $1/2$ supported in $C$ such that $\mathcal{F}_{\alpha}(\rho_1^*) < -N$ (this is possible by the optimality of the classical log-HLS inequality). Define  $\rho^*(\omega) = \rho_1^*(\omega) + \rho_1^*(-\omega)$, then we have
$\int_{\mathbb{S}^2} \omega \rho^*(\omega) \,{\rm d}\omega = 0, \int_{\mathbb{S}^2}  \rho^*(\omega) \,{\rm d}\omega  =1$, 
and $\mathcal{F}_{\alpha}(\rho^*) < -2N + C_0$, where the constant $C_0$ depends only on the distance between $C$ and $-C$. Taking $N \to \infty$ proves that $\mathcal{F}_{\alpha}$ is unbounded from below.

\subsection{Optimal log-HLS inequality on $\R^2$ with constraints}

Let $\mathcal{S}^{-1}:\s^2\setminus\{(0,0,1)\}\to \R^2$ be the stereographic projection on $\R^2$ with respect to the North  pole $\{(0,0,1)\}$, given by
\begin{align*}
\mathcal{S}^{-1}(\omega_1,\omega_2,\omega_3)=
\begin{cases}
\frac{((\omega_1,\omega_2)}{1-\omega_3}\quad \quad &\mbox{if}\ \omega_3<1,\\
\infty\quad \quad &\mbox{if} \ \omega_3=1.
\end{cases}
\end{align*}

Then it's inverse $\mathcal{S}:\R^2\to \s^2\setminus\{(0,0,1)\} $ is given by,
\begin{align*}
\mathcal{S}(x)=\left(\frac{2x_1}{1+|x|^2},\frac{2x_2}{1+|x|^2},\frac{|x|^2-1}{1+|x|^2}\right), \ x = (x_1,x_2).
\end{align*}
It can be shown that by direct computation
\begin{align}\label{sterest}
|\mathcal{S}(x)-\mathcal{S}(y)|=\frac{2|x-y|}{\sqrt{(1+|x|^2)(1+|y|^2)}},
\end{align}
for all $x,y \in \R^2.$ The Jacobian of $\mathcal{S}$ as a map from $\R^2$ to $\R^3$ is $\left(\frac{2}{1 + |x|^2}\right)^2$ 
and hence for any integrable function $\phi: \s^2\to \R$, we have,
\begin{align*}
\int_{\s^2}\phi(\omega) \dw=\frac{1}{\pi}\int_{\R^2} \phi\circ \mathcal{S}(x)\frac{1}{(1+|x|^2)^2}{\rm d} x.
\end{align*}

Given $\rho \in \overset{\circ}{\Gamma}$ (defined on $\s^2$), we define $\varrho (x) := \frac{1}{\pi(1 + |x|^2)} \rho \circ \mathcal{S}(x)$ (defined on $\R^2$). Then $\varrho$ satisfies 
\begin{align*}
\int_{\R^2} \varrho(x) \ {\rm d}x = 1, \ \ \int_{\R^2} \mathcal{S}(x)\varrho(x) \ {\rm d}x = 0.
\end{align*}

We define
\begin{align*}
\overset{\circ}{\Gamma}&:=\left \{\rho\in L^1_{+}(\R^2):\int_{\R^2}\rho\ln \rho \ {\rm d} x <\infty,\int_{\R^2}|x|^2\rho \ {\rm d} x  <\infty, \int_{\R^2} \mathcal{S}(x) \rho(x) \ {\rm d} x = 0 \right \}.
\end{align*}
It is enough to take $\int_{\R^2}\rho\ln(1+|x|^2) \ {\rm d} x  <\infty$ instead of the second moment finiteness, however, to reduce the notational complexity, we defined as it is.
As a consequence of Theorem \ref{main1}, we have the following theorem:

\begin{Th} \label{loghlsr2}
For every $\alpha\geq \frac{1}{2}$, there exist a constant $C(\alpha)\in \R$ such that
\begin{align*}
\int_{\R^2}\varrho\ln \varrho \ {\rm d}x+\frac{2}{\alpha}\int_{\R^2}\int_{\R^2}\varrho(x)\ln|x-y|\varrho(y) \ {\rm d}x{\rm d}y +\left(2-\frac{2}{\alpha}\right)\int_{\R^2}\varrho \ln(1+|x|^2){\rm d} x \geq - C(\alpha),
\end{align*}
holds for every $\varrho \in \overset{\circ}{\Gamma}$.
\end{Th}

\begin{proof}
Given $\rho \in \Gamma_0,$ define $\varrho \in \overset{\circ}{\Gamma}$ by $\varrho (x) := \frac{1}{\pi(1 + |x|^2)} \rho \circ \mathcal{S}(x)$. Then 

\begin{align}\label{abc1}
\int_{\s^2}\rho\ln\rho\dw&=\frac{1}{\pi}\int_{\R^2} \rho \circ \mathcal{S}(x)\ln(\rho\circ \mathcal{S}(x))\frac{1}{(1+|x|^2)^2}{\rm d}x \notag\\
&=\int_{\R^2}\varrho(x)\ln \left(\pi(1+|x|^2)^2)\varrho(x)\right){\rm d} x \notag \\
&=\int_{\R^2}\varrho(x)\ln \varrho(x)\ {\rm d} x
+2\int_{\R^2} \varrho(x)\ln(1+|x|^2){\rm d}x+ \ln \pi.
\end{align}

Now lets simplify the interaction term. The Green's function is given by $G(\omega,\eta)=-2\ln|\omega-\eta|+ c$.  Therefore,
\begin{align}\label{abc2}
&-\int_{\s^2}\int_{\s^2} G(\omega,\eta)\left(\rho(\omega)-1\right)\left(\rho(\eta)-1\right)\dw {\rm d} \eta \notag\\
=&2\int_{\s^2}\int_{\s^2} \ln|\omega-\eta|\left(\rho(\omega)-1\right)\left(\rho(\eta)-1\right)\dw {\rm d} \eta \notag\\
=&2\int_{\s^2}\int_{\s^2}\ln |\omega-\eta|\rho(\omega)\rho(\eta)\dw \rm d \eta -4\int_{\s^2}\int_{\s^2}\ln|\omega-\eta|\rho(\omega) \dw\rm d \eta +2 \int_{\s^2}\int_{\s^2}\ln|\omega-\eta|\dw \rm d \eta \notag \\ 
=&2\int_{\s^2}\int_{\s^2}\ln |\omega-\eta|\rho(\omega)\rho(\eta)\dw \rm d \eta -2\int_{\s^2}\ln|\omega-e_1| \dw,
\end{align}
where we used $ \int_{\s^2}\ln|\omega- \eta| \dw$ is independent of the point $\eta$ by rotational invariance. 
Let us transform the first integral in terms of integrals on $\R^2$. Using \eqref{sterest} we get
\begin{align}\label{abc3}
&\int_{\s^2}\int_{\s^2}  \ln|\omega-\eta|\rho(\omega)\rho(\eta)\ \dw {\rm d}\eta \notag\\
=&\int_{\R^2}\int_{\R^2}\varrho(x)\ln|\mathcal{S}(x)-\mathcal{S}(y)|\varrho(y)\ {\rm d} x {\rm d} y \notag\\
=&\int_{\R^2}\int_{\R^2}\varrho(x) \ln|x-y| \varrho(y){\rm d} x {\rm d} y +\ln 2-\int_{\R^2}\ln(1+|x|^2)\varrho(x)\ {\rm d} x.
\end{align}

Therefore, combining \eqref{abc1}, \eqref{abc2}, \eqref{abc3} we get

\begin{align*}
&\int_{\s^2}\rho \ln \rho {\dw}
- \frac{1}{\alpha}\int_{\s^2}\int_{\s^2} G(\omega,\eta)\left(\rho(\omega)-1\right)\left(\rho(\eta)-1\right)\dw {\rm d} \eta \\
= & \int_{\R^2}\varrho\ln \varrho \ {\rm d}x+\frac{2}{\alpha}\int_{\R^2}\int_{\R^2}\varrho(x)\ln|x-y|\varrho(y) \ {\rm d}x{\rm d}y +\left(2-\frac{2}{\alpha}\right)\int_{\R^2}\varrho \ln(1+|x|^2){\rm d} x + C(\alpha).
\end{align*}
  By Theorem \ref{main1}, we conclude the proof.
\end{proof}

\subsection{Log-HLS inequality under higher moment constraints}

In this final part of this section, following the idea of \cite{CH-2022}, we improve $\alpha$ by imposing additional moment constraints. Recall from introduction 
$ \mathscr{P}_k\vert_{\mathbb{S}^2}=\bigoplus_{i=0}^{k} \mathring{\mathscr{H}}_i\vert_{\mathbb{S}^2},$
and it forms a Hilbert space with respect to the $L^2$ norm on $\mathbb{S}^2$.  

We first recall a few elementary notions of the theory of spherical harmonics and refer to \cite{BarrySimon} for a proof of the properties used below. 
Let $\gamma_m$ denote the dimension of $\mathscr{P}_m$, then it is known that $\gamma_{m}= (m+1)^2$.
Let $\left\{ \Phi_1, \Phi_2, \cdots , \Phi_{\gamma_m}  \right\}$ be an orthonormal basis of $\mathscr{P}_m$ and define the zonal harmonics
    \begin{align*}
        K_m(\omega,\eta)=\sum_{i=1}^{\gamma_m}\Phi_i(\omega)\Phi_i(\eta).
    \end{align*}

Then $K_m(x,y)$ satisfies the following properties \cite[Theorem 3.5.9]{BarrySimon} (note that we are adding all the zonal harmonics up to degree $m$): 

  \begin{itemize}
        \item[(a)] $K_m(x,y)$ is independent of the choice of the orthonormal basis.
        \item[(b)] $K_m(x,x) = \gamma_m = \mbox{dim} \ \mathscr{P}_m.$  
        \item[(c)] $\int_{\s^2} f(\omega) K_m(\omega, \eta) \dw =f(\eta)$ for every $\eta \in \mathbb{S}^2$ and $f \in \mathscr{P}_m$.
        \item[(d)] $K_m(\cdot,p) \in \mathscr{P}_m$ for each fixed $p \in \mathbb{S}^2$.
  \end{itemize}
 
We denote
    \begin{align*}
        \Gamma^m_0= \left \{ \rho \in L^1_{+}(\mathbb{S}^2) \ | \  \int_{\mathbb{S}^2}\rho \text{ln}\rho \dw < \infty, \int_{\mathbb{S}^2}\rho \dw= 1 , \int_{\mathbb{S}^2} p\rho \dw = 0 \ \forall  \ p \in \mathring{\mathscr{P}}_m \right\}.
    \end{align*}

We also denote by $\left[ \frac{m}{2} \right]$ the largest integer less than or equal to $\frac{m}{2}.$
We prove the following non-sharp log-HLS inequality.
\begin{Th} \label{restatement main3}
  Let $\alpha > \frac{1}{\gamma_{\tiny{\left[ \frac{m}{2} \right]}}}$, then  $\inf_{\rho \in \Gamma_m} \mathcal{F}_{\alpha} > -\infty$.
\end{Th}

\begin{proof}
    \textbf{Step 1:}  We proceed as before via contradiction. Suppose that there exists a sequence $\{ \rho_n \}$ in $\Gamma^m_0$ such that $\mathcal{F}_{\alpha}(\rho_n) \rightarrow -\infty$ and set $u_n=\int_{\mathbb{S}^2}G \star (\rho_n - 1)$ and $\bar u_n=0$. Following the same argument, we can conclude that $\mathcal{H}(\rho_n), \mathcal{E}(\rho_n) \rightarrow \infty$. \\

    \textbf{Step 2} $\colon$ By  Gibbs Variational Principle applied to $Mu_n,$ where $M > \frac{1}{\alpha}$ will be fixed later, we have
    \begin{align*}
        \text{ln} \int_{\mathbb{S}^2}e^{Mu_n} \geq \left( M- \frac{1}{\alpha} \right)m_{n}^2 \ .
    \end{align*}
  Proceeding exactly as in \cite{CH-2022}  we conclude
    \begin{align*}
        w_n \rightharpoonup w \text{ in } H^1(\mathbb{S}^2), \ \vert \nabla_{\s^2} w_n \vert^{2} \overset{\star}{\rightharpoonup} \vert \nabla_{\s^2} w \vert^2 + \sigma\text{ in measure}.
    \end{align*}
    and
    \begin{align*}
        \frac{e^{Mu_n}}{\int_{\mathbb{S}^2}e^{Mu_n}} \overset{\star}{\rightharpoonup} \nu= \sum_{i=1}^{N} \nu_{i}\delta_{x_i} \ \text{ in measure}, \ \nu_{i} \geq 0, \ \sum_{i=1}^{N} \nu_{i}=1,
    \end{align*}
    where the blow-up set is characterized by
        \begin{align*}
        \{x_1,x_2 \cdots, x_N\} = \left\{ x \in \mathbb{S}^2 \colon \sigma( \{x \}) \geq \frac{4}{M^2} \left(M - \frac{1}{\alpha} \right) \right\}.
    \end{align*}

    \textbf{Step 3} $\colon$ Let $\rho_n \overset{\star}{\rightharpoonup} \mu$ \text{in measure}. Repeating the same arguments of Theorem \ref{proof alpha half}, Step 3, we conclude  
    \begin{align}\label{blowupmass}
        \text{supp}(\nu) \subseteq \left\{ \omega \in \mathbb{S}^2 \colon \  \mu(\{\omega\}) \geq \frac{1}{M} \right\}. 
    \end{align}
Moreover, $\int_{\s^2} p \ {\rm d}\mu = 0$, for every $p \in \mathring{\mathscr{P}}_m.$

\medskip

    \textbf{Step 4} Now we claim that $\mu(\{\omega\}) \leq \frac{1}{\gamma_{\left [ \frac{m}{2} \right]}}$, for all $\omega \in \s^2$ and for every $\mu$ satisfying such moment constraints.
    For any fixed $p \in \mathbb{S}^2$, since $K_m(\cdot , p) \in \mathscr{P}_m$, it implies $K_m(\cdot ,p)^2 \in \mathscr{P}_{2m}$.
    
     For simplicity of notation we denote $K(\cdot,p) := K_{\left [ \frac{m}{2} \right]}(\cdot, p) \in \mathscr{P}_{\left[ \frac{m}{2} \right]}$. Now assume that $\mu$ has a Dirac mass at $p$. Then 
    \begin{align*}
        \int_{\mathbb{S}^2}K(x,p)^2 \ {\rm d}\mu > \mu(\{p\}) K(p,p)^2  = \mu(\{p\}) (\gamma_{\left[ \frac{m}{2} \right]})^2.
    \end{align*}
    On the other hand, since $\int_{\s^2} p \ {\rm d}\mu = 0$, for every $p \in \mathring{\mathscr{P}}_m,$ and $K(x,p)^2 - \int_{\mathbb{S}^2}K(\omega,p)^2 \ {\rm d}\omega \in  \mathring{\mathscr{P}}_m,$
    \begin{align*}
        \int_{\mathbb{S}^2}K(x,p)^2 \ {\rm d}\mu = \int_{\mathbb{S}^2}K(\omega,p)^2 \ {\rm d}\omega = K(p,p)= \gamma_{\left[ \frac{m}{2} \right]}.
    \end{align*}
  Comparing the last two inequalities we get
    \begin{align*}
        \mu(\{p\}) \gamma_{\left[ \frac{m}{2} \right]}^2 < \gamma_{\left[ \frac{m}{2} \right]}, \ \ \mbox{i.e.}, \ \ 
     \mu(\{p\}) < \frac{1}{\gamma_{\left[ \frac{m}{2} \right]}}.
    \end{align*}
    We fix M such that  $\alpha > \frac{1}{M} \geq \frac{1}{\gamma_{\left[ \frac{m}{2} \right]}}$, we conclude that the blow up set \eqref{blowupmass} is empty.  This contradiction proves $\F_{\alpha}$ is bounded from below for $\alpha > \frac{1}{\gamma_{\left[ \frac{m}{2} \right]}}.$
\end{proof}

\section{Constrained Wasserstein flow} \label{Wflow section}

In this section, we first formally introduce the Wasserstein gradient flow under constraint, following the seminal work of Otto \cite{Otto}. We need a little digression to introduce the essential terminologies on optimal transport. 

\subsection{Optimal Transport and the Wasserstein Distance} 
Let $\mathcal{P}(\mathbb{R}^d)$ be the space of probability measures on $\mathbb{R}^d$, $\mathcal{P}_2(\mathbb{R}^d)$ be the subset of measures with finite second moments, and $\mathcal{P}_{\tiny{ac},2}(\mathbb{R}^d)$ be the subset of absolutely continuous measures with respect to the Lebesgue measure.

 Given a measurable map $T: \mathbb{R}^d \to \mathbb{R}^d$, we say that $T$ transports $\mu$ onto $\nu$, denoted by $\nu = T_{\#}\mu$, if for any measurable set $B \subset \mathbb{R}^d$, $\nu(B) = \mu(T^{-1}(B))$. Equivalently, for any bounded continuous function $\zeta \in C_b^0(\mathbb{R}^d)$:

\begin{align*}
\int_{\mathbb{R}^d} \zeta(T(x)) \, d\mu(x) = \int_{\mathbb{R}^d} \zeta(y) \, d\nu(y).
\end{align*}

The 2-Wasserstein distance between $\mu, \nu \in \mathcal{P}_2(\mathbb{R}^d)$ is defined by the Kantorovich relaxation:
\begin{align*}
W_2^2(\mu, \nu) := \min_{\theta \in \Pi(\mu, \nu)} \int_{\mathbb{R}^d \times \mathbb{R}^d} \vert{}x - y\vert{}^2 \, d\theta(x, y)
\end{align*}

where $\Pi(\mu, \nu)$ represents the set of all transference plans (joint probability measures on $\mathbb{R}^d \times \mathbb{R}^d$ with marginals $\mu$ and $\nu$). When $\mu \in \mathcal{P}_{\tiny{ac},2}(\mathbb{R}^d)$, Brenier's theorem \cite{Brenier}  states that the optimal transference plan is unique and given by the gradient of a convex function: $\theta_{\mbox{\tiny{opt}}} = (\text{I}_{\mathbb{R}^d} , \nabla\varphi)_{\#}\mu$. In other words, the Wasserstein distance can be expressed as:
\begin{align*}
 W_2^2(\mu, \nu) = \inf_{\nu=T_{\#}\mu} \int_{\mathbb{R}^d} \vert{}x - T(x)\vert{}^2 \, d\mu(x) = \int_{\mathbb{R}^d} \vert{}x - \nabla\varphi(x)\vert{}^2 \, d\mu(x). 
\end{align*}

A highly useful property of the $W_2$ distance is its ability to control averages. If $\zeta$ is a Lipschitz function with Lipschitz constant $\mbox{Lip}(\zeta)$, then for any $\mu, \nu \in \mathcal{P}_2(\mathbb{R}^d)$:

\begin{align} \label{Lip control}
\left\vert{} \int_{\mathbb{R}^d} \zeta(x) \, d\mu - \int_{\mathbb{R}^d} \zeta(x) \, d\nu \right\vert{} \leq \mbox{Lip}(\zeta) \, W_2(\mu, \nu).
\end{align}
See \cite[Lemma 2.4]{JKO2} for a proof. For a comprehensive theory of optimal transport, Wasserstein gradient flows and its applications to other fields of mathematics we refer to the monographs \cite{Villani03, Villani09, San, AGS} and more recent texts \cite{ABSOT, FigalliOT, MaggiOT} on this subject.

\medskip

Let $\F : \mathcal{P}(\mathbb{R}^d) \to \mathbb{R} \cup \{\pm \infty\}$ be a functional that is finite at least on a subset of $\mathcal{P}(\mathbb{R}^d),$ say $\F$ is finite on the set of all absolutely continuous measure with finite entropy and second moment.  Then the Wasserstein gradient flow of $\F$ is given by 
\begin{align*}
\partial_t \rho = \nabla \cdot (\rho \nabla \Phi), \ \ \ \ \Phi = \F^{\prime}(\rho).
\end{align*}
It is easy to check at a formal level that the energy $\F(\rho(t))$ decreases along the flow.
Now let $\bm{V} : \R^d \to \R^m$ be a vector field. We are interested in a flow that preserves the moment constraint 
\begin{align*}
\int_{\R^d} \bm{V}(x) \rho(x,t) \ {\rm d}x = 0,
\end{align*}
and the energy $\mathcal{F}$ decreases along the flow. To determine the associated flow, we introduce a time-dependent Lagrange multiplier: Let
\begin{align*}
\tilde \Phi = \Phi + \bm{\lambda}(t) \cdot \bm{V}(x) =  \F^{\prime}(\rho) + \bm{\lambda}(t) \cdot \bm{V}(x)
\end{align*}
and consider the associated flow
\begin{align*}
\partial_t \rho = \nabla \cdot (\rho \nabla \tilde\Phi).
\end{align*}
The $\bm{\lambda}(t)$ is chosen so that the moment constraint is preserved along the flow: 
\begin{align*}
0  = \frac{d}{dt} \int \int_{\R^d} \bm{V}(x) \rho(x,t) \ {\rm d}x = \int_{\R^d} \bm{V}(x) \partial_t\rho(x,t) \ {\rm d}x 
= \int_{\R^d}\bm{V}(x) \nabla \cdot (\rho \nabla \tilde\Phi) \ {\rm d}x.
\end{align*}
Writing, in components $\bm{V} = (V_1, \ldots, V_m),$ and $\bm{\lambda}(t) = (\lambda_1(t), \ldots, \lambda_m(t))$ for every $i \in \{1,\ldots,m\}$
\begin{align*}
0 &=  \int_{\R^d}V_i(x) \nabla \cdot (\rho \nabla \tilde\Phi) \ {\rm d}x \\
&= - \int_{\R^d}\rho \nabla V_i(x) \cdot \nabla \F^{\prime}(\rho) \ {\rm d}x  - \int_{\R^d}\rho \nabla V_i(x) \cdot \nabla (\bm{\lambda}(t) \cdot \bm{V}(x)) \ {\rm d}x \\
&= - \int_{\R^d}\rho \nabla V_i(x) \cdot \nabla \F^{\prime}(\rho) \ {\rm d}x  - \sum_{k = 1}^m \lambda_k(t) \int_{\R^d} \rho \nabla V_i(x) \cdot \nabla V_k(x) \ {\rm d}x.
\end{align*}
We define the $m \times m$ symmetric matrix $\mathcal{A}[\rho(t)]$ and the column vector  $\bm{b}[\rho(t)]$ (surpassing the dependance of $t$ in the integral) by 
\begin{align*}
\mathcal{A}[\rho(t)]_{ik} = \int_{\R^n} \rho \nabla V_i(x) \cdot \nabla V_k(x) \ {\rm d}x, \ \ \mbox{and} \ \ \bm{b}[\rho(t)]_i = \int_{\R^n}\rho \nabla V_i(x) \cdot \nabla \F^{\prime}(\rho) \ {\rm d}x.
\end{align*}
Assume that the matrix $\mathcal{A}[\rho(t)]$ is invertible for every $t >0,$ then the Lagrange multiplier is uniquely determined by
\begin{align*}
\bm{\lambda}(t) = - \mathcal{A}[\rho(t)]^{-1} (\bm{b}[\rho(t)]).
\end{align*}
For simplicity of notations, sometimes, we will drop time $t$ from the notations.
Therefore, the associated system we are interested in is the following:
\begin{align}
\begin{cases}
&\partial_t \rho = \nabla \cdot (\rho \nabla [ \F^{\prime}(\rho) + \bm{\lambda}(t) \cdot \bm{V}(x)]) \\
&\mathcal{A}[\rho] \bm{\lambda} = - \bm{b}[\rho],
\end{cases}
\end{align}
where 
\begin{align}\label{a and b}
\mathcal{A}[\rho]_{ij} = \int_{\R^d} \rho \nabla V_i \cdot \nabla V_j\ {\rm d}x, \ \ \ \bm{b}[\rho]_i = \int_{\R^d} \rho \nabla V_i \cdot \nabla \F^{\prime}(\rho) \ {\rm d}x.
\end{align}

\subsection{Energy dissipation inequality}
We claim that 
\begin{align*}
\frac{d}{dt} \F(\rho(t)) = - \int_{\R^n} \rho |\nabla \tilde \Phi|^2 \ {\rm d}x.
\end{align*}
Indeed, 
\begin{align*}
\frac{d}{dt} \F(\rho(t)) &= \int_{\R^d} \F^{\prime}(\rho) \partial_t \rho \ {\rm d}x \\
& =  \int_{\R^d} \F^{\prime}(\rho) \nabla \cdot (\rho \nabla \tilde\Phi)  \ {\rm d}x\\
& = - \int_{\R^d} \rho \nabla \F^{\prime}(\rho)  \nabla \tilde\Phi \ {\rm d}x \\
& =  - \int_{\R^d} \rho \nabla (\tilde \Phi - \bm{\lambda} \cdot \bm{V})  \cdot  \nabla \tilde\Phi \ {\rm d}x \\
& = - \int_{\R^d} \rho |\nabla \tilde \Phi|^2 \ {\rm d}x +  \int_{\R^d} \rho \nabla (\bm{\lambda} \cdot \bm{V}) \cdot \nabla \tilde \Phi \ {\rm d}x.
\end{align*}
On the other hand, 
\begin{align*}
\int_{\R^d} \rho \nabla (\bm{\lambda} \cdot \bm{V}) \cdot \nabla \tilde \Phi \ {\rm d}x &= \sum_{k = 1}^m \lambda_k(t)\int_{\R^d} \rho \nabla V_k \cdot \nabla (\F^{\prime}(\rho) + \bm{\lambda} \cdot \bm{V}) \ {\rm d}x \\
& =  \sum_{k = 1}^m \lambda_k(t)\int_{\R^d} \rho \nabla V_k \cdot \nabla \F^{\prime}(\rho) + \sum_{k=1}^m \sum_{j = 1}^m \lambda_j(t) \lambda_k(t) \int_{\R^d}  \rho \nabla V_i \cdot \nabla V_j\ {\rm d}x \\
& =  \bm{b}[\rho] \cdot \bm{\lambda} + \bm{\lambda}^{tr} \mathcal{A}[\rho] \bm{\lambda} = 0.
\end{align*}
This proves the energy dissipation inequality.

\subsection{A constrained Patlak-Keller-Segel system on $\mathbb{R}^2$} \label{pks subsection}
We now concentrate on the particular case of  the log-HLS energy functional on $\mathbb{R}^2$, and for simplicity, we remove the dependence of $\chi$ from the notations.
\begin{align}\label{loghls}
\F(\rho) &= \inte \rho(x) \ln \rho(x) \ {\rm d}x + \frac{\chi}{4\pi} \inte \inte \rho(x) \ln |x-y| \rho(y) \ {\rm d}x{\rm d}y \\
&=: \mathcal{H}(\rho) + \frac{\chi}{4\pi} \mathcal{W}(\rho), \notag
\end{align}
and the vector field we consider is the stereographic coordinates, 
\begin{align*}
\mathcal{S}(x) = \frac{1}{1 + |x|^2}\left(2x_1, 2x_2, |x|^2 - 1\right).
\end{align*}

Let $\Gamma$ denote the set
\begin{align*}
\Gamma := \left\{\rho \in L^1_+(\mathbb{R}^2) \ | \ \int_{\mathbb{R}^2} \rho\ln \rho \ {\rm d}x < \infty, \int_{\mathbb{R}^2} \rho \ {\rm d}x = 1, \int_{\mathbb{R}^2}|x|^2 \rho \ {\rm d}x<\infty\right\},
\end{align*}
and let 
\begin{align*}
\overset{\circ}{\Gamma} := \left\{\rho \in \Gamma \ | \ \int_{\mathbb{R}^2} \mathcal{S}(x)\rho(x) \ {\rm d}x = 0\right\}.
\end{align*}
By Theorem \ref{loghlsr2}, the functional $\F$ is bounded from below on $\overset{\circ}{\Gamma} $ if $2 \leq \frac{\chi}{4\pi} \leq 4,$ i.e. $8\pi \leq \chi \leq 16\pi.$ The associated Patlak-Keller-Segel system takes the form
\begin{equation}\label{mKS}
\begin{cases}
\partial_t \rho(x,t) = \Delta \rho - \chi \nabla \cdot \left[\rho(x,t) \nabla\left(u(x,t) + \frac{1}{\chi}\bm{\lambda}(t) \cdot \mathcal{S}(x)\right) \right] \\ 
u(x,t) := - \frac{1}{2\pi}\int_{\mathbb{R}^2}\ln |x - y| \rho(y,t) {\rm d}y\\ 
\mathcal{A}[\rho(t)](\bm{\lambda}(t)) = -\bm{b}[\rho(t)], \ \mbox{with initial data} \  \rho(x,0) = \rho_0(x) \in \overset{\circ}{\Gamma},
\end{cases}
\end{equation}
 where $\mathcal{A}[\rho]$ and $\bm{b}[\rho]$ are given by \eqref{a and b} with associated $\F$ given by \eqref{loghls} and $\bm{V}$ given by the stereographic projection. In particular,
\begin{align*}
\left(\mathcal{A}[\rho]\right)_{ij} = \int_{\mathbb{R}^2}\rho \nabla \mathcal{S}_i \cdot \nabla \mathcal{S}_j \ {\rm d}x. 
\end{align*}
It is easy to see that the matrix $\mathcal{A}[\rho]$
is positive definite for every $\rho \in \overset{\circ}{\Gamma}$. Indeed, lifting the integral to $\s^2,$ and writing  
$\varrho(\omega) = \rho(\mathcal{S}^{-1}(\omega))\pi(1 + |\mathcal{S}^{-1}(\omega)|^2)^2$ we can rewrite the integral as (up to a constant multiple)
\begin{align}\label{the A rho}
\mathcal{A}[\rho] &= \int_{\s^2} (I_{\R^3} - \omega \otimes \omega)(1 - \omega_3)^2 \varrho(\omega){\rm d}\omega.
\end{align}
Clearly, $\mathcal{A}[\rho]$ is a positive definite matrix as $\rho \geq 0$ and its integral is equal to $1$. In addition, if we have uniform entropy and second moment bound on $\rho$ (and hence uniform entropy bound on $\varrho$), then the matrix $\mathcal{A}[\rho]$ is uniformly bounded from below by $cI$ with $c>0$ (see Lemma \ref{a priori estimates}[d] below). Indeed, if for a sequence $\rho_m$, we have uniform entropy and moment bound, while one of the eigenvalue of  $\mathcal{A}[\rho_m]$ converging to zero, then up to a subsequence, $\rho_m$ converges weakly in $L^1$ to some $\rho$ with integral equal to $1.$ Since the matrix $\mathcal{A}[\rho]$ is weakly $L^1$ continuous, we get the desired contradiction.

The $\bm{b}[\rho]_i = \int_{\R^2} \rho \nabla \mathcal{S}_i \cdot \nabla \F^{\prime}(\rho) \ {\rm d}x,$ in this case is given by after a formal integration by parts, and using $\nabla u(x) = -\frac{1}{2\pi} \int_{\R^2} \frac{(x - y)}{|x - y|^2} \rho(y) \ {\rm d}y,$
\begin{align} \label{the b rho}
 \bm{b}[\rho]_i  &= \int_{\R^2} \rho \nabla \mathcal{S}_i \cdot \nabla (1 + \ln \rho - \chi u) \ {\rm d}x \notag\\
 &= \int_{\R^2}  \nabla \mathcal{S}_i \cdot \nabla  \rho - \chi \int_{\R^2} \rho \nabla \mathcal{S}_i \cdot \nabla  u \ {\rm d}x \notag\\
 &= -\int_{\R^2}   \rho\Delta \mathcal{S}_i   + \frac{\chi}{4\pi} \int_{\R^2} \int_{\R^2} (\nabla \mathcal{S}_i(x) - \nabla \mathcal{S}_i(y)) \cdot \frac{x - y}{\vert{}x - y\vert{}^2} \rho(x) \rho(y) \,{\rm d}x \,{\rm d}y.
\end{align}
The last line follows by using the symmetry in $x$ and $y$.
Since $\nabla \mathcal{S}_i(x)$ and $\Delta \mathcal{S}_i(x)$ are bounded, $\bm{b}[\rho]$ is well defined. 

\subsection{Weak formulation}

\begin{Def} \label{weak formulation}
For any initial data $\rho_0 \in \overset{\circ}{\Gamma}$ and $T^*>0$, we say that a non-negative  function $\rho \in (C([0,T^*); L^1_{\text{\tiny{weak}}}(\R^2))$ is a weak solution to the PKS-system \eqref{mKS} on the time interval $(0,T^*)$ if for all test function $\psi \in C_c^{\infty}(\R^2)$, $t \mapsto \int_{\R^2} \psi(x)\rho(x,t) \ {\rm d}x$ is absolutely continuous and
\begin{align*}
\frac{d}{dt} \int_{\R^2} \psi(x)\rho(x,t) \ {\rm d}x = &\int_{\R^2} \Delta \psi(x)\rho(x,t) \ {\rm d}x \\
& - \frac{\chi}{4\pi} \int_{\R^2}\int_{\R^2} (\nabla \psi(x) - \nabla \psi(y)) \cdot \frac{(x - y)}{|x - y|^2} \rho(x, t)\rho(y,t) \ {\rm d}x {\rm d}y\\
& + \int_{\R^2} \nabla \psi(x) \cdot \nabla (\bm{\lambda}(t) \cdot \mathcal{S}(x))\rho(x,t) \ {\rm d}x
\end{align*}
together with the conditions $\rho(x,0) = \rho_0(x)$ and $\mathcal{A}[\rho(t)] \bm{\lambda}(t) =  \bm{b}[\rho(t)]$ for all $t\in (0, T^*)$ hold, where $\mathcal{A}[\rho]$ and $\bm{b}[\rho]$ are given by \eqref{the A rho} and \eqref{the b rho} respectively. If $T^* = \infty$ we say that $\rho$ is a global weak solution of the system.
\end{Def}

\section{The constrained JKO scheme}\label{jko section}

As in the classical case, we consider the minimizing movement scheme or the JKO scheme, but with constraint: for a time step $\tau>0$, we define recursively
\begin{align} \label{JKO scheme}
\rho^n_{\tau}
\in \arg\min_{\rho \in  \overset{\circ}{\Gamma}}\Big\{
\mathcal{F}(\rho) + \frac{1}{2\tau}W_2^2(\rho,\rho^{n-1}_{\tau})
\Big\}, \ \ n \geq 1,
\end{align}
with $\rho_{\tau}^0 = \rho_0.$  If $\chi < 16 \pi,$ then by Lemma \ref{well defined} below, the JKO-scheme \eqref{JKO scheme} is well defined. The goal is the show that an appropriate interpolation of the $\rho_{\tau}^n$s converge to a weak solution of \eqref{mKS}. 

\medskip

\begin{Lem}[Existence of minimizers]\label{well defined}
\begin{itemize}
\item [(a)] Assume $\chi < 16\pi$, and fix $\bar{\rho} \in \overset{\circ}{\Gamma}$ and $\tau > 0$. Then, the functional 
\begin{align*}
 \rho \mapsto \mathcal{F}(\rho) + \frac{1}{2\tau} W_2^2(\rho, \bar{\rho})
\end{align*}
is bounded from below on $\overset{\circ}{\Gamma}$ and is sequentially lower semicontinuous with respect to the weak topology of $L^1(\mathbb{R}^2)$.
\item[(b)] If $\chi < 16\pi$, then all the sub-level sets $\{\mathcal{F}(\rho) + \frac{1}{2\tau}  W_2^2(\rho, \bar{\rho}) \leq C\}$ are sequentially precompact with respect to the weak topology of $L^1(\mathbb{R}^2)$. In particular, the minimization problem $\inf_{\rho \in \overset{\circ}{\Gamma}} \left(\mathcal{F}(\rho) + \frac{1}{2\tau}  W_2^2(\rho, \bar{\rho})\right)$ admits a solution.
\end{itemize}
\end{Lem}

The boundedness from below in (a) follows from Theorem \ref{loghlsr2} (for $\chi \leq 8\pi$ we use the classical log-HLS inequality). The remaining of the proof of the above lemma is classical, see for example in this particular context  \cite[Lemma 2.1 and Lemma 3.1]{JKO2}. This later parts require the hypothesis $\chi < 16\pi$.

\medskip

Next we present the a priori estimates on the interpolates, their interpolations and convergence results. These results are classical, so we only mention it in the following lemmas.

\medskip

We define the time interpolation of the discrete scheme by connecting each pair of consecutive densities $\rho_\tau^n$ and $\rho_\tau^{n+1}$ with the constant-speed Wasserstein geodesic on the interval $[n\tau,(n+1)\tau)$. Thus, for every
$t\in[n\tau,(n+1)\tau),$ let $\nabla\varphi^n$ denote the Brenier map (the optimal transport map) sending $\rho_\tau^n$ to $\rho_\tau^{n+1}$ \cite{Brenier}. Then,  we define
\begin{align} \label{time interpolates}
\rho_\tau(t) = \left(\frac{(n+1)\tau-t}{\tau}\,\mathrm{Id} + \frac{t-n\tau}{\tau}\,\nabla\varphi^n \right)_{\#}\rho_\tau^n.
\end{align}
This yields the constant-speed geodesic joining $\rho_\tau^n$ and $\rho_\tau^{n+1}$ in the Wasserstein space \cite{McCann}. In particular, $\rho_{\tau}(n\tau) = \rho_{\tau}^n$ and for $t\in[n\tau,(n+1)\tau]$
\begin{align*}
W_2\bigl(\rho_\tau^n,\rho_\tau(t)\bigr)
=
\frac{t-n\tau}{\tau}\,
W_2\bigl(\rho_\tau^n,\rho_\tau^{n+1}\bigr).
\end{align*}

\begin{Lem}[A priori estimates] \label{a priori estimates}
Let $\chi < 16 \pi$. Then for every $T \in (0, \infty),$ there exists a constant $C(T),$ such that for each $\tau \in (0, 1)$ and positive integers $N$ satisfying $N\tau \leq T$, the following estimates hold
\begin{itemize}
\item[(a)]
\begin{align*}
\frac{1}{2\tau} \sum_{n=1}^N W_2^2(\rho_\tau^n, \rho_\tau^{n-1}) + \sup_{1 \leq n \leq N}\int_{\R^2}|x|^2\rho_\tau^n \ {\rm d}x +  \sup_{1 \leq n \leq N}\int_{\mathbb{R}^2} \rho_{\tau}^n \vert{}\ln \rho_{\tau}^n\vert{} \ {\rm d}x \leq C(T).
\end{align*}
\item[(b)] For every $0 \leq t \leq N\tau \leq T,$ 
\begin{align*}
\sup_{0 \leq t \leq N\tau}\int_{\R^2}|x|^2\rho_\tau(t) \ {\rm d}x + \sup_{0 \leq t \leq N\tau}\int_{\mathbb{R}^2} \rho_{\tau}(t) \vert{}\ln \rho_{\tau}(t)\vert{} \ {\rm d}x \leq C(T).
\end{align*}
\item[(c)] There exists a constant $C_0$ depending only on (up to the third derivative of) $\mathcal{S}$ such that for every $t \in [0,T],$ and $k \in \mathbb{N}$ satisfying $t \in [k\tau, (k + 1)\tau],$ the following estimates hold
\begin{align*}
&\left|\int_{\R^2} \mathcal{S}(x)\rho_{\tau}(t) \ {\rm d}x\right| \leq C_0 W_2(\rho_{\tau}(t), \rho_{\tau}^k) \leq C_0\sqrt{\tau}, \\
&\left|\mathcal{A}[\rho_{\tau}(t)] - \mathcal{A}[\rho_{\tau}^k]\right| \leq C_0 W_2(\rho_{\tau}(t), \rho_{\tau}^k) \leq C_0\sqrt{\tau}, \\
&\left|\bm{b}[\rho_{\tau}(t)] - \bm{b}[\rho_{\tau}^k]\right| \leq C_0 W_2(\rho_{\tau}(t), \rho_{\tau}^k)\leq C_0\sqrt{\tau}.
\end{align*}
\item[(d)] For every $C_{\star} > 0,$ there exists $c_{\star}>0$ such that for every $\rho \in \overset{\circ}{\Gamma}$ satisfying uniform entropy and second moment bound bound $\mathcal{H}(\rho) \leq C_{\star}, \int_{\R^2}|x|^2 \rho(x) \ {\rm d}x \leq C_{\star}$ 
\begin{align*}
\mathcal{A}[\rho] \geq c_{\star} \mbox{I}_{\R^3},
 \end{align*}
holds, where $\mbox{I}_{\R^3}$ is the identity matrix in $\R^3$.
\end{itemize}
\end{Lem}

\begin{proof} For (a), we simply test \eqref{JKO scheme} against $\rho_{\tau}^{n-1}:$  $\mathcal{F}(\rho_{\tau}^{n}) + \frac{1}{2\tau}W_2^2(\rho_{\tau}^{n},\rho^{n-1}_{\tau}) \leq \mathcal{F}(\rho_{\tau}^{n-1}),$ for $1 \leq n \leq N$. Adding all the terms, and using $\mathcal{F}$ is bounded from below on $\overset{\circ}{\Gamma}$ if $\chi = 16\pi,$ and denoting $\inf_{\overset{\circ}{\Gamma}} \mathcal{F}_{16\pi} = C_{\mbox{\tiny{log-HLS}}}(16\pi)$ we get 
\begin{align*}
\frac{\chi}{16\pi}C_{\mbox{\tiny{log-HLS}}}(16\pi) + \left(1 - \frac{\chi}{16\pi}\right)\int_{\R^2} \rho_{\tau}^{N} \ln \rho_{\tau}^{N} \ {\rm d}x + \frac{1}{2\tau} \sum_{n=1}^N W_2^2(\rho_\tau^n, \rho_\tau^{n-1}) \leq \mathcal{F}(\rho_0).
\end{align*}
The second moment bound is an easy consequence of the Wasserstein distance bound. Finally, the desired entropy bound bound follows from the Carleman estimate \cite[Lemma 2.6]{BD}:
\begin{align*}
\int_{\mathbb{R}^2} \rho \vert{}\log \rho\vert{} \, {\rm d}x \le \int_{\mathbb{R}^2} \rho \left( \log \rho + \vert{}x\vert{}^2 \right) {\rm d}x + 2 \log(2\pi)  + \frac{2}{e}.
\end{align*}
Part (b) follows from McCann's displacement convexity of the second moment and the entropy \cite{McCann} (see also Villani's monograph \cite[Theorem 5.15]{Villani03}). Part (c) follows from \eqref{Lip control}.

Finally, for part (d), we argue by contradiction. Suppose that the statement is false. Then there exists a sequence ${\rho_k}$ satisfying the uniform entropy and second-moment bounds such that one of the eigenvalue of $\mathcal{A}[\rho_k] \rightarrow 0.$ By compactness, up to passing to a subsequence, $\rho_k\rightharpoonup \rho$ weakly in $L^1$, where $\rho$ has finite entropy and finite second moment. By the weak $L^1$ continuity of $\mathcal{A}$, we obtain one of the eigenvalue of $\mathcal{A}[\rho] = 0$ which is impossible.
\end{proof}

\begin{Lem}[Convergence] \label{convergence}
 Let $\rho_0 \in \overset{\circ}{\Gamma}$ and $\chi < 16\pi, T>0$, and $\rho_{\tau}(t)$ be defined as in \eqref{time interpolates} for $t \leq T.$ Then
\begin{itemize}
\item[(a)]  For any $0 \leq s \leq t \leq T$, the curve $\rho_\tau$ satisfies the $1/2$-H\"older estimate in the Wasserstein space:
\begin{align*}
W_2(\rho_\tau(s), \rho_\tau(t)) \leq \sqrt{6 \left( \mathcal{F}(\rho_0) - \inf_{ n \in \mathbb{N}} \mathcal{F}(\rho_{\tau}^n) \right)} (t - s)^{1/2}.
\end{align*}

\item[(b)] Thus, after extracting a subsequence,
\begin{align*}
\rho_\tau \longrightarrow \rho \ \ \text{in } \  C\left([0, T]; L^1_{\text{\tiny{weak}}}(\mathbb{R}^2)\right), \ \ \mbox{as} \ \tau \rightarrow 0.
\end{align*}
In particular, for every $\psi \in C_{\text{b}}(\mathbb{R}^2)$,
\begin{align*}
\sup_{t \in [0, T]} \left\vert{} \int_{\mathbb{R}^2} \psi(x) \left( \rho_\tau(x, t) - \rho(x,t) \right) {\rm d}x \right\vert{} \longrightarrow 0, \ \ \mbox{as} \ \tau \rightarrow 0.
\end{align*}

\item[(c)] For every fixed $t \in [0, T]$, the weak $L^1$ convergence $\rho_\tau(t) \rightharpoonup \rho(t)$ implies convergence of $\rho_\tau(t) \otimes \rho_\tau(t) \rightharpoonup \rho(t) \otimes \rho(t)$ weakly in $L^1(\mathbb{R}^{2} \times \R^2),$
in the sense that
\begin{align*}
\int_{\mathbb{R}^{2}} \int_{\mathbb{R}^{2}}\Psi(x, y) \rho_\tau(x,t) \rho_\tau(y,t) \, {\rm d}x \, {\rm d}y \longrightarrow \int_{\mathbb{R}^{2}}\int_{\mathbb{R}^{2}} \Psi(x, y) \rho(x,t) \rho(y,t) \, {\rm d}x \, {\rm d}y
\end{align*}
for every bounded continuous function $\Psi \in C_{\text{b}}(\mathbb{R}^{2} \times \mathbb{R}^{2})$ as $\tau \rightarrow 0$.  Moreover, since the convergence of $\rho_\tau$ is uniform in $t$, this convergence is also uniform in time:
\begin{align*}
\sup_{t \in [0, T]} \left\vert{}\int_{\mathbb{R}^{2}} \int_{\mathbb{R}^{2}}\Psi(x, y) (\rho_\tau(x,t) \rho_\tau(y,t) - \rho(x,t) \rho(y,t)) \, {\rm d}x \, {\rm d}y \right\vert{} \longrightarrow 0.
\end{align*}

\item[(d)]  In addition, the following convergences
\begin{align*}
\int_{\R^2} \mathcal{S}(x)\rho_{\tau}(t) \ {\rm d}x \rightarrow \int_{\R^2} \mathcal{S}(x)\rho(t)  \ {\rm d}x , \ 
\mathcal{A}[\rho_{\tau}(t)] \rightarrow \mathcal{A}[\rho(t)], \ 
\bm{b}[\rho_{\tau}(t)]\rightarrow \bm{b}[\rho(t)],
\end{align*}
hold uniformly in $t \in [0,T],$ as $\tau \rightarrow 0.$ 
In particular,   $\int_{\R^2} \mathcal{S}(x)\rho(t)  \ {\rm d}x = 0,$ and
\begin{align*}
\mathcal{A}[\rho_{\tau}(t)]^{-1}\bm{b}[\rho_{\tau}(t)] \rightarrow \mathcal{A}[\rho(t)]^{-1}\bm{b}[\rho(t)] \ \ \ \mbox{in} \ C[0,T].
\end{align*}
\end{itemize}
\end{Lem}
\begin{proof}
For a proof of (a), we refer to \cite[(equation (3.5))]{JKO2}.  The convergence are all classical, we refer to \cite{JKO, JKO2} for this particular context.
\end{proof}

\subsection{The first variation formula}
Given $\rho_{\tau}^n \in \overset{\circ}{\Gamma}$ and let $\rho_{\tau}^{n+1} \in \overset{\circ}{\Gamma}$ be defined as in the JKO-scheme \eqref{JKO scheme}. For simplicity of notations, in this subsection we denote $\rho_{\tau}^{n+1}$ by $\rho.$
 
 \medskip
 
Given a test function $\psi \in C_c^{\infty}(\R^2),$ define the (transport) map:
\begin{equation}
T_{\varepsilon,\bm{a}}(y) = y+\varepsilon\nabla\psi(y)
+\sum_{j=1}^3a_j\nabla \mathcal{S}_j(y),
\end{equation}
where $\bm{a}=(a_1,a_2,a_3)\in\mathbb R^3$ will be chosen below, depending on $\varepsilon$, such that 
$\rho_{\varepsilon} := (T_{\varepsilon, \bm{a}(\varepsilon)})_\#\rho$ is an admissible competitor for every $\varepsilon$ sufficiently small.  For $\delta >0$, define the smooth map $F: (-\delta, \delta) \times \R^3 \to \R^3,$ by
\begin{equation}
F(\varepsilon,\bm{a})
:=
\int_{\mathbb R^2}
\mathcal{S}(T_{\varepsilon, \bm{a}}(y))\rho(y) \,{\rm d}y.
\end{equation} 
Since $\rho$ is admissible, $F(0,0)=0$. The differential of this constraint map with respect to $\bm{a} = (a_1,a_2,a_3)$ at $(0,0)$ has entries:
\begin{align*}
\left. \frac{\partial F_i}{\partial a_j} \right|_{(0,0)} = 
\int_{\mathbb R^2}
\nabla S_i(y)\cdot\nabla S_j(y)\rho(y) \,{\rm d}y = \mathcal{A}[\rho]_{ij},
\end{align*}
which is invertible. Then the implicit function theorem gives a $C^1$-function
$\bm{a}:(-\delta,\delta)\to\mathbb R^3$ such that $\bm{a}(0)=0,$ and $F(\varepsilon,\bm{a}(\varepsilon))=0.$ Therefore $\rho_\varepsilon
:=
(T_{\varepsilon,\bm{a}(\varepsilon)})_\#\rho
\in \overset{\circ}{\Gamma}$ for every sufficiently small $\varepsilon$. Differentiating $F(\varepsilon, \bm{a}(\varepsilon)) = 0$ with respect to $\varepsilon$ and evaluating at $\varepsilon=0$ yields:
\begin{equation}
\int_{\mathbb R^2} D\mathcal{S} \, \nabla\psi\rho \,{\rm d}y
+
\mathcal{A}[\rho] \bm{a}^{\prime}(0) = 0,
\end{equation}
and hence we obtain
\begin{align} \label{aprime}
\bm{a}^{\prime}(0)=-\mathcal{A}[\rho]^{-1} \left(\int_{\mathbb R^2} D\mathcal{S} \, \nabla\psi\rho \,{\rm d}y\right).
\end{align}
Note that $\bm{a}^{\prime}(0)$ is linear in $\psi$.
We also denote the vector field $\xi(y) := \left. \frac{d}{d\varepsilon} \right|_{\varepsilon=0} T_{\varepsilon,\bm{a}(\varepsilon)}(y)$ 
which is given by
\begin{align}\label{effective vector field}
\xi(y) = \nabla\psi(y) - \sum_{j=1}^3 \bm{a}^{\prime}_j(0)\nabla \mathcal{S}_j(y),
\end{align}
and hence, it is also linear in $\psi.$

\medskip

\noindent
{\bf{First variation of the entropy:}}
Since for $\varepsilon$ small, the map $T_{\varepsilon,\bm{a}(\varepsilon)} : \mathbb{R}^2 \to \R^2$ is a $C^1$ diffeomorphism, the classical change-of-variables formula gives
\begin{align*}
\int\rho_\varepsilon\log\rho_\varepsilon = 
\int\rho\log\rho - \int\rho\log\det DT_{\varepsilon,\bm{a}(\varepsilon)}.
\end{align*}
Differentiating this and evaluating at $\varepsilon=0$ gives
\begin{equation}
\left. \frac{d}{d\varepsilon} \right|_{\varepsilon=0} \int\rho_\varepsilon\log\rho_\varepsilon = 
-\int_{\mathbb R^2}
\operatorname{div}\xi \, \rho \,{\rm d}x.
\end{equation}
Plugging in \eqref{effective vector field}, we get
\begin{align} \label{entropy 1st var}
\left. \frac d{d\varepsilon} \right|_{\varepsilon=0} \int\rho_\varepsilon\log\rho_\varepsilon
={}&
-\int \Delta\psi\rho \,{\rm d}y + \sum_{j=1}^3\bm{a}_j'(0) \int\Delta \mathcal{S}_j\rho \,{\rm d}y.
\end{align}

\medskip

\noindent
{\bf{ First variation of the interaction energy.}}
By definition of the push forward measure $\rho_\varepsilon = (T_{\varepsilon,\bm{a}(\varepsilon)})_\#\rho$, we have
\begin{align*}
\mathcal{W}(\rho_\varepsilon) = \int_{\mathbb R^2}\int_{\mathbb R^2} \log\left| T_{\varepsilon,\bm{a}(\varepsilon)}(x) - T_{\varepsilon,\bm{a}(\varepsilon)}(y) \right| \rho(x)\rho(y) \,{\rm d}x \,{\rm d}y.
\end{align*}
Using \eqref{effective vector field} and expanding in Taylor series $T_{\varepsilon,\bm{a}(\varepsilon)}(x) = x + \varepsilon\xi(x) + o(\varepsilon)$, the kernel expands as
\begin{align*}
&\log\left| T_{\varepsilon,\bm{a}(\varepsilon)}(x) - T_{\varepsilon,\bm{a}(\varepsilon)}(y) \right| 
=\log\left|(x-y) + \varepsilon(\xi(x)-\xi(y)) + o(\varepsilon)\right| \\
= &\log|x-y| + \varepsilon \frac{(x-y)\cdot(\xi(x)-\xi(y))}{|x-y|^2} + o(\varepsilon).
\end{align*}
Therefore, we get
\begin{align*}
\left.\frac{d}{d\varepsilon}\right|_{\varepsilon=0} \mathcal{W}(\rho_\varepsilon) &= \int_{\mathbb R^2}\int_{\mathbb R^2} \frac{(x-y)\cdot(\xi(x)-\xi(y))}{|x-y|^2} \rho(x)\rho(y) \,{\rm d}x \,{\rm d}y. 
\end{align*}

Plugging the the expression of $\xi$ from \eqref{effective vector field}, we get

\begin{align}\label{first var W}
\left.\frac{d}{d\varepsilon}\right|_{\varepsilon=0} \mathcal{W}(\rho_\varepsilon) =& \int_{\mathbb R^2}\int_{\mathbb R^2} \frac{(x-y)\cdot(\nabla \psi(x)-\nabla \psi(y))}{|x-y|^2} \rho(x)\rho(y) \,{\rm d}x \,{\rm d}y \notag \\
& - \sum_{j = 1}^3 \bm{a}_j^{\prime}(0) \int_{\mathbb R^2}\int_{\mathbb R^2} \frac{(x-y)\cdot(\nabla \mathcal{S}_j(x)-\nabla \mathcal{S}_j(y))}{|x-y|^2} \rho(x)\rho(y) \,{\rm d}x \,{\rm d}y.
\end{align}

\medskip

\noindent
{\bf{ First variation of the Wasserstein distance.}}
Let $\nabla\varphi_{\#}\rho_{\tau}^n = \rho$ be the Brenier map, then the map $T_{\varepsilon,\bm{a}(\varepsilon)}\circ\nabla\varphi$ pushes $\rho_{\tau}^n$ forward to $\rho_\varepsilon$. Hence,
\begin{align*}
W_2^2(\rho_{\tau}^n,\rho_\varepsilon)
\leq
\int_{\mathbb R^2}
\left|
T_{\varepsilon,\bm{a}(\varepsilon)}(\nabla\varphi(x))-x
\right|^2
\rho_{\tau}^n(x) \,{\rm d}x.
\end{align*}
Again by Taylor series, $T_{\varepsilon,\bm{a}(\varepsilon)}(y) = y+\varepsilon\xi(y)+o(\varepsilon),$ we obtain
\begin{align} \label{W2 1st var}
&\left. \frac {d}{d\varepsilon} \right|_{\varepsilon=0} \frac{1}{2} \int \left| T_{\varepsilon,\bm{a}(\varepsilon)}(\nabla\varphi(x))-x
\right|^2 \rho_{\tau}^n(x) \,{\rm d}x \notag\\
 =&\int_{\mathbb R^2} \bigl(\nabla\varphi(x)-x\bigr) \cdot \xi(\nabla\varphi(x)) \rho_{\tau}^n(x) \,{\rm d}x, \notag\\
 =& \int (\nabla\varphi(x)-x) \cdot\nabla\psi(\nabla\varphi(x)) \rho_{\tau}^n(x) \,{\rm d}x  -\sum_{j=1}^3 \bm{a}_j'(0) \int (\nabla\varphi(x)-x)
\cdot\nabla S_j(\nabla\varphi(x)) \rho_{\tau}^n(x) \,{\rm d}x.
\end{align}

\subsection{The Euler-Lagrange equation}
Given $\rho_{\tau}^n \in  \overset{\circ}{\Gamma},$ let $\rho_{\tau}^{n+1} \in \overset{\circ}{\Gamma}$ be a minimizer, and let
 $\rho_{\varepsilon}$ is defined as above with $\rho$ replaced by $\rho_{\tau}^{n+1} $.
 Let $\rho^{n+1}_{\tau} = (\nabla\varphi^n)_\# \rho^n_{\tau}$ be the Brenier map. Then by minimality of $\rho_{\tau}^{n+1},$ we have
 \begin{align*}
 \mathcal{F}(\rho_{\varepsilon}) + \frac{1}{2\tau} W_2^2(\rho_{\varepsilon}, \rho_{\tau}^n) \geq  \mathcal{F}(\rho_{\tau}^{n+1}) + \frac{1}{2\tau} W_2^2(\rho_{\tau}^{n+1}, \rho_{\tau}^n),
 \end{align*}
 and hence,
\begin{align*}
\frac{\left(\mathcal{H}(\rho_{\varepsilon}) - \mathcal{H}(\rho_{\tau}^{n+1})\right)}{\varepsilon}  + \frac{\chi}{4\pi} \frac{\left(\mathcal{W}(\rho_{\varepsilon}) - \mathcal{W}(\rho_{\tau}^{n+1})\right)}{\varepsilon} + \frac{1}{2\tau} \frac{W_2^2(\rho_{\varepsilon}, \rho_{\tau}^n) - W_2^2(\rho_{\tau}^{n+1}, \rho_{\tau}^n)}{\varepsilon} \geq 0,
\end{align*}
for every $\varepsilon >0$ sufficiently small. Letting $\varepsilon \rightarrow 0$, using \eqref{entropy 1st var}, \eqref{first var W}, \eqref{W2 1st var}, and recalling every term is linear in $\psi$, we conclude
\begin{align} \label{elequation}
 &\ \int_{\mathbb R^2} (\nabla\varphi^n(x) - x) \cdot \nabla \psi(\nabla\varphi^n(x)) \, \rho_{\tau}^n(x) \,{\rm d}x \notag\\
 ={} & \tau \int_{\mathbb R^2} \Delta \psi(x) \rho_{\tau}^{n+1}(x) \,{\rm d}x  - \frac{\chi}{4\pi} \tau \int_{\mathbb R^2} \int_{\mathbb R^2} (\nabla \psi(x) - \nabla \psi(y)) \cdot \frac{x - y}{\vert{}x - y\vert{}^2} \rho_{\tau}^{n+1}(x) \rho_{\tau}^{n+1}(y) \,{\rm d}x \,{\rm d}y \notag \\
 & \, \, + \tau\sum_{j=1}^3 \boldsymbol{a}^{\prime}_j(0) (\bm{b}_{\tau}^{n+1})_j
 \end{align}
where $\bm{b}_{\tau}^{n+1} := ((\bm{b}_{\tau}^{n+1})_1 ,(\bm{b}_{\tau}^{n+1})_2 ,(\bm{b}_{\tau}^{n+1})_3 )$ is given by

\begin{align} \label{bexpression}
(\bm{b}_{\tau}^{n+1})_j :={}& -\int_{\mathbb R^2} \Delta \mathcal{S}_j(x) \rho_{\tau}^{n+1}(x) \,{\rm d}x \notag \\ & + \frac{\chi}{4\pi} \int_{\mathbb R^2} \int_{\mathbb R^2} (\nabla \mathcal{S}_j(x) - \nabla \mathcal{S}_j(y)) \cdot \frac{x - y}{\vert{}x - y\vert{}^2} \rho_{\tau}^{n+1}(x) \rho_{\tau}^{n+1}(y) \,{\rm d}x \,{\rm d}y \notag\\ & + \frac{1}{\tau} \int_{\mathbb R^2} (\nabla\varphi^n(x) - x) \cdot \nabla \mathcal{S}_j(\nabla\varphi^n(x)) \, \rho_{\tau}^n(x) \,{\rm d}x, \notag \\
& = \bm{b}[\rho_{\tau}^{n+1}]_j + O\left(\frac{1}{\tau} W_2^2(\rho_{\tau}^{n+1}, \rho_{\tau}^n)\right)
\end{align}
where recall $\bm{b}[\rho]$ is given by \eqref{the b rho} and the last term is of $ O\left(\frac{1}{\tau} W_2^2(\rho_{\tau}^{n+1}, \rho_{\tau}^n)\right)$ which follows from the constraints as described below.

Since $\rho^{n}_{\tau}, \rho^{n+1}_{\tau} \in  \overset{\circ}{\Gamma}$ and $\rho^{n+1}_{\tau} = (\nabla\varphi^n)_\# \rho^n_{\tau}$, their moment constraints give:
\begin{align*}
\int_{\mathbb{R}^2} \left[ \mathcal{S}_j(\nabla\varphi^n(x)) - \mathcal{S}_j(x) \right] \rho^{n}_{\tau} \, {\rm d}x = 0.
\end{align*}

Applying a second-order Taylor expansion to $\mathcal{S}_j(x)$ around $\nabla\varphi^n(x)$ and integrating against $\rho^{n}_{\tau}$ yields 
\begin{align*}
\int_{\mathbb R^2} (\nabla\varphi^n(x) - x) \cdot \nabla \mathcal{S}_j(\nabla\varphi^n(x)) \rho_{\tau}^n(x) \, {\rm d}x = O \left(\int_{\mathbb R^2} |\nabla\varphi^n(x) - x|^2 \rho^{n}_{\tau} \, {\rm d}x \right) = O\left(W_2^2(\rho_{\tau}^{n+1}, \rho_{\tau}^n)\right).
\end{align*}

We define, for $0 \leq t \leq T,$
\begin{align*}
\bm{\lambda}_{\tau}(t) = -\mathcal{A}[\rho_{\tau}(t)]^{-1}\bm{b}[\rho_{\tau}(t)].
\end{align*}

Note that by Lemma \ref{a priori estimates}(a)(d), and \eqref{bexpression} 
\begin{align*}
\tau \mathcal{A}[\rho_{\tau}^{n+1}]^{-1} \left(\bm{b}_{\tau}^{n+1} - \bm{b}[\rho_{\tau}(t)]\right) = O\left(W_2^2(\rho_{\tau}^{n+1}, \rho_{\tau}^n)\right).
\end{align*}

By symmetry of  $\mathcal{A}[\rho_{\tau}^{n+1}]$,  we can simplify the last term in \eqref{elequation} as follows:
\begin{align} \label{bexpression1}
 \tau\sum_{j=1}^3 \boldsymbol{a}^{\prime}_j(0) (\bm{b}_{\tau}^{n+1})_j
 &= \tau\left( -\mathcal{A}[\rho_{\tau}^{n+1}]^{-1} \int_{\mathbb{R}^2} D\mathcal{S}(x) \nabla\psi(x) \rho_{\tau}^{n+1}(x) \, {\rm d}x \right) \cdot \bm{b}_{\tau}^{n+1} \notag\\ 
 &= - \tau\left( \int_{\mathbb{R}^2} D\mathcal{S}(x) \nabla\psi(x) \rho_{\tau}^{n+1}(x) \, {\rm d}x \right) \cdot \left( \mathcal{A}[\rho_{\tau}^{n+1}]^{-1} \bm{b}_{\tau}^{n+1} \right)\notag \\ 
 &= - \tau\left( \int_{\mathbb{R}^2} D\mathcal{S}(x) \nabla\psi(x) \rho_{\tau}^{n+1}(x) \, {\rm d}x \right) \cdot \Big[-\bm{\lambda}_{\tau}((n+1)\tau)\notag \\
 & \hspace{7 cm} + O\Big(\frac{1}{\tau} W_2^2(\rho_{\tau}^{n+1}, \rho_{\tau}^n)\Big)\Big]\notag \\ 
 &=  \tau\sum_{j=1}^3 (\bm{\lambda}_{\tau}((n+1)\tau))_j \int_{\mathbb{R}^2} \nabla \mathcal{S}_j(x) \cdot \nabla \psi(x) \rho_{\tau}^{n+1}(x) \, {\rm d}x \notag\\
 & \hspace{8 cm} + O\left(W_2^2(\rho_{\tau}^{n+1}, \rho_{\tau}^n)\right)\notag \\ 
 &=  \tau\int_{\mathbb{R}^2} \left( \sum_{j=1}^3 ((\bm{\lambda}_{\tau}((n+1)\tau))_j  \nabla \mathcal{S}_j(x) \right) \cdot \nabla \psi(x) \rho_{\tau}^{n+1}(x) \, {\rm d}x \notag\\
 & \hspace{8 cm} + O\left(W_2^2(\rho_{\tau}^{n+1}, \rho_{\tau}^n)\right)\notag\\ 
 &= \tau\int_{\mathbb{R}^2} \nabla(\bm{\lambda}_{\tau}((n+1)\tau) \cdot \mathcal{S}(x)) \cdot \nabla\psi(x) \rho_{\tau}^{n+1}(x) \, {\rm d}x + O\left(W_2^2(\rho_{\tau}^{n+1}, \rho_{\tau}^n)\right).
\end{align}

\medskip

\noindent
{\bf Proof of Theorem \ref{main4}:}

\begin{proof}
By Taylor’s expansion:
\begin{align}
\psi(\nabla\varphi^n(x)) - \psi(x) =  (\nabla\varphi^n(x) - x) \cdot \nabla\psi \circ \nabla\varphi^n(x) + O\left(\vert{}x - \nabla\varphi^n(x)\vert{}^2\right)
\end{align}
 and hence
we can recast the left-hand side of \eqref{elequation} as:
\begin{align} \label{elequation1}
& \int_{\mathbb R^2} (\nabla\varphi^n(x) - x) \cdot \nabla \psi(\nabla\varphi^n(x)) \, \rho_{\tau}^n(x) \,{\rm d}x \notag\\
=& \left( \int_{\mathbb{R}^2} \psi \circ \nabla\varphi^n(x) \rho^n_\tau(x) \, {\rm d}x - \int_{\mathbb{R}^2} \psi(x) \rho^n_\tau(x) \, {\rm d}x \right) + O\left( \int_{\mathbb{R}^2} \vert{}x - \nabla\varphi^n(x)\vert{}^2 \rho^n_\tau(x) \, {\rm d}x \right) \notag \\ 
=& \left( \int_{\mathbb{R}^2} \psi(x) (\rho^{n+1}_\tau(x) \, {\rm d}x -  \rho^n_\tau(x)) \, {\rm d}x \right) + O\left( W^2_2(\rho^n_\tau, \rho^{n+1}_\tau) \right).
\end{align}
Combining \eqref{elequation}, \eqref{bexpression1}, and \eqref{elequation1}
\begin{align} \label{elequation2}
 &\int_{\mathbb{R}^2} \psi(x) (\rho^{n+1}_\tau(x) - \rho^n_\tau(x) )  {\rm d}x + O\left( W^2_2(\rho^n_\tau, \rho^{n+1}_\tau) \right)  \notag  \\ 
 =& \tau \int_{\mathbb{R}^2}  \Delta\psi(x) \rho^{n+1}_\tau(x) \, {\rm d}x - \frac{\chi}{4\pi} \tau\int_{\mathbb{R}^2}\int_{\mathbb{R}^2} (\nabla\psi(x) - \nabla\psi(y)) \cdot \frac{(x - y)}{\vert{}x - y\vert{}^2} \rho^{n+1}_\tau(x) \rho^{n+1}_\tau(y) \, {\rm d}x \, {\rm d}y 
 \notag \\
 & +\tau\int_{\mathbb{R}^2} \nabla(\bm{\lambda}_{\tau}((n+1)\tau) \cdot \mathcal{S}(x)) \cdot \nabla\psi(x) \rho_{\tau}^{n+1}(x) \, {\rm d}x.
\end{align}

Thanks to \eqref{Lip control}, for $n\tau \leq t \leq (n+1)\tau,$ and $\zeta \in C_c^{\infty}(\R^2)$
\begin{align*}
\tau \int_{\R^2} \zeta(x) \rho_{\tau}^{n+1}(x) \ {\rm d}x = \int_{n\tau}^{(n+1)\tau} \int_{\R^2} \zeta(x) \rho_{\tau}(x,t) \ {\rm d}x \ {\rm d}t + O(W_2(\rho_{\tau}^{n+1}, \rho_{\tau}^{n})).
\end{align*}
Similarly, for $n\tau \leq t \leq (n+1)\tau,$ Lemma \ref{a priori estimates}[(c)] gives
\begin{align*}
\bm{\lambda}_{\tau}((n+1)\tau) = \bm{\lambda}_{\tau}(t) + O(W_2(\rho_{\tau}^{n+1}, \rho_{\tau}^{n})).
\end{align*}

Let $0 \leq t_1 < t_2 \leq T$ and denote $m_1 = [\frac{t_1}{\tau}] + 1, m_1 = [\frac{t_2}{\tau}],$ then summing \eqref{elequation2} over $n = m_1, \ldots, m_2,$ the bordering integrals $t_1$ to $m_1\tau$ and $m_2\tau$ to $t_2$ being negligible, we deduce
\begin{align*}
&\int_{\mathbb{R}^2} \psi(x) (\rho_\tau(x, t_2) - \rho_\tau(x,t_1) )  {\rm d}x   \notag  \\ 
 =& \int_{t_1}^{t_2} \int_{\mathbb{R}^2}  \Delta\psi(x) \rho_{\tau}(x, t) \, {\rm d}x dt \\
 &  - \frac{\chi}{4\pi} \int_{t_1}^{t_2}\int_{\mathbb{R}^2}\int_{\mathbb{R}^2} (\nabla\psi(x) - \nabla\psi(y)) \cdot \frac{(x - y)}{\vert{}x - y\vert{}^2} \rho_\tau(x,t) \rho_\tau(y,t) \, {\rm d}x \, {\rm d}y \, dt
 \notag \\
 & +\int_{t_1}^{t_2}\int_{\mathbb{R}^2} \nabla(\bm{\lambda}_{\tau}(t) \cdot \mathcal{S}(x)) \cdot \nabla\psi(x) \rho_{\tau}(x,t) \, {\rm d}x \, dt + O(\sqrt{\tau}),
\end{align*}
together with the constraint
\begin{align*}
\bm{\lambda}_{\tau}(t) = - \mathcal{A}[\rho_{\tau}(t)]^{-1}\bm{b}[\rho_{\tau}(t)].
\end{align*}

Thanks to Lemma \ref{convergence}, we can pass to the limit in the above equations, showing that the limit $\rho(x, t)$ is a weak solution to \eqref{mKS} in the sense of Definition \ref{weak formulation}. This concludes the proof.
\end{proof}

\medskip

\medskip

 \noindent
{\bf Acknowledgement.}  D. Karmakar acknowledges the support of the Department of Atomic Energy, Government of India,
under Project Identification No. RTI 4014 and ANRF ARG-MATRICS grant ANRF/ARGM/2025/000976/MTR. M. Ghosh acknowledges the postdoctoral fellowship   provided  through ARG project  grant ANRF/ARG/2025/000572/MS DT.12.03.2026 under the mentorship of Professor Debdip Ganguly at Indian Statistical Institute, Delhi.
The authors gratefully acknowledge that this work was significantly enhanced during their participation in the in-person programs {\bf Geometric Analysis and PDE} [ICTS/GPDE/2026/02] and {\bf A School on Gradient Flows} [ICTS/FLOWS/2026/08], hosted by the International Centre for Theoretical Sciences (ICTS) in February and August–September 2026 respectively.

\medskip

\noindent
{\bf Disclosure Statement.} The authors report there are no competing interests to declare.

\medskip

\noindent
{\bf Data availability statement.} Data sharing not applicable to this article as no datasets
were generated or analysed during the current study.

\bibliographystyle{alpha}
\bibliography{Onofri.bib}

 \end{document}